\documentclass[12pt]{article}

\title{\bf \Large Minimax-optimal estimation for sparse multi-reference alignment with collision-free signals}
\author{ {Subhro Ghosh} \thanks{Authors are listed in the alphabetical order of their surnames.} 
\thanks{Department of Mathematics, National University of Singapore (\texttt{subhrowork@gmail.com})}
\and{Soumendu Sundar Mukherjee}  \footnotemark[1] \thanks{Statistics and Mathematics Unit, Indian Statistical Institute, Kolkata (\texttt{ssmukherjee@isical.ac.in})} 
\and{Jing Bin Pan}  \footnotemark[1] \thanks{Department of Mathematics, National University of Singapore, (\texttt{jingbin.pan1@gmail.com})} \thanks{Corresponding author: Jing Bin Pan (\texttt{jingbin.pan1@gmail.com})} }

\date{}

\usepackage[usenames,dvipsnames]{xcolor}
\usepackage{hyperref}
\hypersetup{
    colorlinks=true,
    linkcolor=cyan!80!black,
    citecolor=MidnightBlue,
    urlcolor=magenta,
}

\usepackage{titling}
\usepackage{adjustbox}
\usepackage{amsmath}
\usepackage{amssymb}
\usepackage{amsthm}
\usepackage{apacite}
\usepackage{bbm}
\usepackage{chngcntr}
\usepackage{enumitem}
\usepackage[e]{esvect}
\usepackage[margin=0.8in]{geometry}
\usepackage{graphicx}
\usepackage{mathrsfs} 
\usepackage{mathtools}
\usepackage[numbers]{natbib}
\usepackage{relsize}
\usepackage{scalerel,stackengine}
\usepackage{tikz-cd}
\usepackage{upgreek}
\usepackage{microtype}
\graphicspath{} 

\makeatletter
\renewcommand\@biblabel[1]{[#1]}%

\renewcommand\NAT@bibsetnum[1]{%
   \setlength{\labelwidth}{1em}%
   \setlength{\labelsep}{2\labelsep}%
   \setlength{\leftmargin}{\labelwidth}
   \addtolength{\leftmargin}{\labelsep}%
   \setlength{\itemsep}{\bibsep}\setlength{\parsep}{\z@}%
   \ifNAT@openbib
     \addtolength{\leftmargin}{\bibindent}%
     \setlength{\itemindent}{-\bibindent}%
     \setlength{\listparindent}{\itemindent}%
     \setlength{\parsep}{0pt}%
   \fi
}

\newcommand*{\numbibliography}{%
  \let\@bibsetup\NAT@bibsetnum
  \let\@biblabel\NAT@biblabelnum
  \bibliography}
\makeatother

\setlist[enumerate,1]{label={(\roman*)}}
\def\[#1\]{\begin{align*}#1\end{align*}}

\newtheoremstyle{mystyle}
  {}
  {}
  {\upshape}
  {}
  {\bfseries}
  {.}
  { }
  {\thmname{#1}\thmnumber{ #2}\thmnote{ (#3)}}

\DeclarePairedDelimiterX{\infdivx}[2]{(}{)}{%
  #1\;\delimsize\|\;#2%
}
\newcommand{\infdiv}{D_\text{KL}\infdivx}

\newcommand{\norm}[1]{\left\lVert#1\right\rVert}
\newcommand{\normnorm}[1]{\lVert#1\rVert}

\newcommand{\opnorm}[1]{\left\lVert#1\right\rVert_\text{op}}
\newcommand{\itbf}[1]{\textbf{\textit{#1}}}
\newcommand{\supp}{\text{Supp}}

\newcommand{\R}{\mathbb{R}}

\newcommand{\Z}{\mathbb{Z}}
\newcommand{\E}{\mathbb{E}}

\newcommand{\bs}{\boldsymbol}

\newcommand{\mc}[1]{\mathcal{#1}}

\newcommand{\RN}[1]{%
  \textup{\uppercase\expandafter{\romannumeral#1}}%
}
\newcommand\reallywidehat[1]{\arraycolsep=0pt\relax%
\begin{array}{c}
\stretchto{
  \scaleto{
    \scalerel*[\widthof{\ensuremath{#1}}]{\kern-.5pt\bigwedge\kern-.5pt}
    {\rule[-\textheight/2]{1ex}{\textheight}} 
  }{\textheight} %
}{0.5ex}\\           
#1\\                 
\rule{-1ex}{0ex}
\end{array}
}

\theoremstyle{mystyle}
\newtheorem{theorem}{Theorem}[subsection]

\newtheorem{lemma}[theorem]{Lemma}
\newtheorem{definition}[theorem]{Definition}

\newtheorem{proposition}[theorem]{Proposition}

\newtheorem{fact}[theorem]{Fact}

\newenvironment{enum}
  {\begin{enumerate}[label = (\roman*)]}
  {\end{enumerate}}

\begin{document}

\maketitle

\begin{abstract} The Multi-Reference Alignment (MRA)  problem aims at the recovery of an unknown signal from repeated observations under the latent action of a group of cyclic isometries, in the presence of additive noise of high intensity $\sigma$. It is a more tractable version of the celebrated cryo EM model. In the crucial high noise regime, it is known that its sample complexity scales as $\sigma^6$. Recent investigations have shown that for the practically significant setting of sparse signals, the sample complexity of the maximum likelihood estimator asymptotically scales with the noise level as $\sigma^4$. 
In this work, we investigate minimax optimality for signal estimation under the MRA model for so-called \textit{collision-free} signals. In particular, this signal class covers the setting of generic signals of \textit{dilute sparsity} (wherein the support size $s=O(L^{1/3})$, where $L$ is the ambient dimension.  

We demonstrate that the minimax optimal rate of estimation in for the sparse MRA problem in this setting is $\sigma^2/\sqrt{n}$, where $n$ is the sample size. In particular, this widely generalizes the sample complexity asymptotics for the restricted MLE in this setting, establishing it as the statistically optimal estimator.     
Finally, we demonstrate a concentration inequality for the restricted MLE on its deviations from the ground truth.
\end{abstract}



\newpage
\tableofcontents

\pagebreak

\section{Introduction}\label{Introduction}

\subsection{Introduction}\label{IntroductionIntroduction}

Multi-reference alignment (MRA) is the problem of estimating an unknown signal from noisy 
samples that are subject to latent rotational transformations. In recent years, the MRA model has found 
practical applications in many scientific fields, such as structural biology \cite{StructuralBiology1} \cite{StructuralBiology2} 
\cite{StructuralBiology3} \cite{StructuralBiology4} \cite{StructuralBiology5}\cite{StructuralBiology6}, image recognition 
\cite{ImageRecognition1} \cite{ImageRecognition2} \cite{ImageRecognition3} \cite{ImageRecognition4}, robotics \cite{Robotics} and 
signal processing \cite{SignalProcessing1} \cite{SignalProcessing2}, and has garnered significant research interest as a result. 
Most notably, the MRA model serves as a simplified model for the three-dimensional reconstruction problem in cryo-electron microscopy (cryo-EM) 
\cite{CryoEM1} \cite{Generic} \cite{CryoEM2}.

In this paper, we study the estimation problem in the MRA model \cite{MRA1}\cite{Generic}\cite{MRA2} for collision-free signals. 
Identify the vector space $\R^L$ with the space of functions $\Z/L\Z \to \R$. Let 
\[\mc{R} := \{ R^{(\ell)} \ : \ 1\leq \ell \leq L\}\]
denote the cyclic group of coordinate shifts, where $R^{(\ell)} : \R^L \to \R^L$ is the linear operator 
given by $(R^{(\ell)}\theta)(i) = \theta(i + \ell)$. In the general MRA estimation problem, 
the objective is to recover an unknown parameter $\theta\in \R^L$, known in the literature as a \itbf{signal}, 
from $n$ independent noisy observations $X_1,X_2,\cdots,X_n$ given by
\begin{align}\label{IntroductionIntroduction1}
X_i = G_i \theta + \sigma\xi_i, \ \ \ \ \ \ \ i \in \{1,\cdots,n\},\end{align}
where the $G_i$'s are drawn from $\mc{R}$ independently and uniformly at random and 
each $\xi_i \sim \mc{N}(\bs 0, \bs{I}_d)$ is i.i.d standard Gaussian noise that is independent of the $R_i$'s. We
remark that while the canonical distribution on $\mc{R}$ is uniform, other distributions have also been considered \cite{OtherDistribution}.

In this model, any vector $\theta$ is statistically indistinguishable from its cyclic shifts $R^{(1)}\theta,\cdots,R^{(L-1)}\theta$. 
Consequently, the parameter $\theta$ is only identifiable up to the action of $\mc{R}$. Hence we study the estimation problem 
under the rotation-invariant distance
\[\rho(\theta,\phi) := \min_{G \in \mc{R}} \norm{\theta - G\phi}.\]
Here, $\norm{\cdot}$ refers to the usual Euclidean norm. 

Both the dimension $L$ and the noise level $\sigma$ are assumed to be known. 
In many practical applications such as cryo-EM, the noise level is very high but can be 
lowered over time by technological improvements \cite{Improvement}. As such, understanding the 
relationship between the difficulty of the estimation problem and the quantity $\norm{\theta}/\sigma$, 
known as the \itbf{signal-to-noise
ratio} is of fundamental importance. With that in mind, we focus our work on understanding the 
asymptotic sampling complexity in the MRA model as $\sigma \to +\infty$. 

The MRA problem, particularly in the low-noise regime, 
has been mostly attacked using the synchronization approach \cite{MRA1}.
However, it was recently recognised as a Gaussian mixture model \cite{BRW}, which enabled the use of various methods such as 
the method of moments \cite{MethodOfMoments} \cite{Generic} and expectation-maximization \cite{ExpectationMaximization}.
For a more detailed discussion on the likelihood landscape of such 
models, we refer the reader to \cite{LikelihoodLandscape1} \cite{LikelihoodLandscape2} \cite{LikelihoodLandscape3} 
\cite{LikelihoodLandscape4}.

\subsection{Estimation rates for MRA and sparsity}\label{IntroductionCollision}

In the two seminal papers by Bandeira, Rigollet and Weed \cite{BRW} \cite{Generic} in 2017, it was shown that the general MRA model in the high 
noise regime has 
a sampling complexity of $O(\sigma^L)$ in the worst case, and $O(\sigma^6)$ in the 
case of signals having full Fourier support. This has motivated investigations into 
variations of the MRA model where a better sampling complexity may be achieved. One such variation 
is the MRA estimation problem for certain classes of sparse signals, where the restricted maximum likelihood estimator (abbrv. MLE) has been shown to exhibit an improved sampling complexity of $O(\sigma^4)$\cite{Sparse}. We refer the reader to \cite{Sparse1} and \cite{Sparse2} for an investigation into algorithmic aspects of sparse MRA and the role of sparsity in the related problem of cryo EM.

While it is of interest to understand the the sample complexity asymptotics for the restricted MLE, it only shades light on the performance of a certain type of estimators. In order to have a complete understanding of the problem, one needs to understand minimax optimal rates of estimation for the model. In other words, the central question is, what is the best possible behaviour over candidate estimators for the worst possible error over a sufficiently rich class of signals. Minimax optimality in statistical problems has a long and rich history; for a detailed account we refer the interested reader to \cite{Tsybakov}. 
In the context of the MRA problem, minimax optimal rate of estimation for the general MRA problem is known to be $\sigma^3/\sqrt{n}$ (c.f. \cite{BRW,Zhou-Zhou2022}). However, minimax optimality for the sparse MRA problem is not understood, and in view of different sample complexity asymptotics, is not covered by the minimax theory on general signal spaces. This is the broad direction that we propose to explore in the work.

\subsection{The space of collision free signals}

In this work, we will focus our attention on a particular class of sparse signals 
called collision-free signals.
For a vector $\theta\in \R^L$, let $\mc{D}(\theta)$ denote the multiset 
of differences
\[\mc{D}(\theta) := \big\{ i - j \ (\text{mod } L) \ : \ i,j\in \text{supp}(\theta) \text{ with } i \neq j\big\}.\]
We say that $\theta$ is \itbf{collision-free} if each element in $\mc{D}(\theta)$ 
appears with multiplicity exactly $1$. 

The notion of collision-freeness is motivated by the beltway problem in combinatorial optimization. 
The beltway problem consists of recovering a set $\mc{S}$ of integers from 
their pairwise differences $\mc{D}_{\mc{S}}$, up to the trivial symmetry of translating 
all the numbers in the set by the same amount. In 
1939, Piccard \cite{Piccard} conjectured that if two sets 
of integers $\mc{S}_1$ and $\mc{S}_2$ have the same set of pairwise differences $\mc{D}$, 
and the pairwise differences are known to be unique (i.e. $\mc{S}_1$ and $\mc{S}_2$ 
are collision-free), then the sets $\mc{S}_1$ and $\mc{S}_2$ must be translates of each other.

A counterexample to Piccard's conjecture was found by Bloom \cite{Bloom} in 1975 for $|\mc{S}| = 6$.
In 2007, a complete resolution was brought about by 
Bekir and Golomb \cite{Golomb}, who showed that Piccard's conjecture holds for $|\mc{S}| \geq 7$. In our current context, 
this means that the support of a collision-free signal $\theta$ can be recovered from $\mc{D}(\theta)$ up to translation as long as $|\supp(\theta)| \geq 7$. This fact will 
play a crucial role in allowing us to obtain improved bounds on the sampling complexity.

With the above discussion in mind, we will assume that 
the unknown parameter $\theta$ satisfies:
\begin{enum}
\item $\theta$ is collision free;
\item $\sigma \geq \norm{\theta}$;
\item There exist positive constants $m,M,\epsilon\in \R_{>0}$ 
such that $m \leq |\theta(i)| \leq M$ for all $i\in\supp(\theta)$ and 
$s := |\supp(\theta)| \geq \max\{7, (2+\epsilon)M^2/m^2\}$.
\end{enum}
Let $\mc{T}$ denote the set of signals satisfying the above three conditions. In our analysis, we will only keep track of the dependence on $\sigma$ and $s$. 
The other quantities $L, m, M$ and $\epsilon$ are fixed throughout the paper 

Collision free signals arise as a natural concept in the setting of sparsity, especially in the context of sparse MRA. It is straightforward to envisage that the sparser a signal, the more likely it is for its support to be a collision free subset of $\Z/L \Z$. Indeed, a randomly chosen subset of $\Z/L \Z$ (with expected size $s$) is going to be collision free with high probability as long as $s=O(L^{1/3})$ (c.f \cite{Sparse} Appendix D). The latter corresponds to the regime of so-called \textit{dilute sparsity}, per the terminology used in the investigations in \cite{Sparse}, and has a natural place in the study of sparse MRA.

\subsection{Main Results}\label{IntroductionMain}

For a signal $\theta$, let $P_\theta$ denote the distribution of a random variable $X$ satisfying
\[X = G\theta  + \sigma\xi,\]
where $G$ is drawn from $\mc{R}$ uniformly and $\xi_i \sim \mc{N}(\bs 0, \bs{I}_d)$. 
Let $f_\theta : \R^d \to \R$ denote the density function of $P_\theta$. Explicitly,
\[f_\theta(x) &= \frac{1}{\sigma^L (2\pi)^{L/2}} \mathbb{E}_G\bigg[ 
\exp\bigg(- \frac{1}{2\sigma^2}\norm{x - G\theta}^2\bigg)\bigg].
\]
Given $n$ i.i.d samples $X_1,\cdots,X_n$ as in (\ref{IntroductionIntroduction1}), we analyse the performance of the restricted maximum 
likelihood estimator (MLE) $\tilde{\theta}_n$ given by
\[\tilde{\theta}_n := \underset{\phi\in \mc{T}}{\text{arg max}} 
\sum_{i=1}^n \log \mathbb{E}_G\bigg[\exp \Big( - \frac{1}{2\sigma^2} \norm{X_i - G\phi}^2\Big)\bigg].\]
Our main results are the following two theorems, which characterises the rate of estimation in the collision-free MRA model.
\begin{theorem}\label{IntroductionMain1}
We have that
\begin{equation}\label{IntroductionMain2}
\inf_{\hat{\theta}_n} \sup_{\theta\in \mc{T}} \mathbb{E}\big[\rho(\hat{\theta}_n,\theta)\big] \asymp\frac{\sigma^2}{\sqrt{n}},
\end{equation}
where the infimum is taken over all estimators $\hat{\theta}_n$ based on $n$ samples $X_1,\cdots,X_n$.
\end{theorem}

\begin{theorem}\label{IntroductionMain3} Fix a constant $\delta\in (0, 2\sqrt{L}M)$. There exist constants $C_s,\tilde{C}_s \in \R_{>0}$ 
depending on $s$ such that for all $n \geq C_s \delta^{-4} \sigma^{12}$ and 
for all $\theta\in \mc{T}$, the restricted MLE $\tilde{\theta}_n$ satisfies
\[P(\rho(\tilde{\theta}_n,\theta) \geq \delta) \leq C_s \sigma^{5s}\delta^{-2s} \exp\bigg(-\frac{\tilde{C}_s n\delta^4}{\sigma^{12}}\bigg).\]
\end{theorem}

Theorem \ref{IntroductionMain1} is a direct consequence of the following 
two results, which we state below for independent interest.

The first result is a uniform upper bound for the restricted MLE of $O(\sigma^4)$, which is a marked improvement over 
the pointwise upper bound in \cite[Theorem 4]{Sparse}.
\begin{theorem}\label{IntroductionMain4}
There exists positive constants $C$ and $C_s$, where $C_s$ depends on $s$, such that the restricted MLE $\tilde{\theta}_n$ satisfies
\begin{equation}\label{IntroductionMain5}
\mathbb{E}\big[\rho(\tilde{\theta}_n,\theta)\big] \leq C \cdot \frac{\sigma^2}{\sqrt{n}} + C\cdot \frac{\sigma^3 \log n }{sn}
+ C_s\cdot \frac{\sigma^{11}}{n}
\end{equation}
uniformly over all choices $\theta\in \mc{T}$ for all $\sigma$ sufficiently large.
\end{theorem}
The second result is a lower bound on the minimax rate of estimation in the 
sparse MRA model. In particular, it shows that the restricted MLE 
achieves the optimal sampling complexity of $O(\sigma^4)$.
\begin{theorem}\label{IntroductionMain6}
Fix a sparsity $s\in\Z_{\geq 1}$. Let $\mc{T}_s$ denote the set of vectors $\theta \in \mc{T}$ satisfying $|\supp(\theta)| = s$.
For any $\sigma \geq \max_{\theta\in \mc{T}_s}\norm{\theta}$, 
there exists a universal constant $M$ such that
\[\inf_{\tilde{\theta}_n}\sup_{\theta\in\mc{T}_s} \mathbb{E}_\theta\big[\rho(\tilde{\theta}_n,\theta)\big] \geq M\min\bigg\{\frac{\sigma^2}{\sqrt{n}},\ 1 \bigg\},\]
where the infimum is taken over all estimators $\tilde{\theta}_n$ on $n$ samples from $P_\theta$.
\end{theorem}

In section \ref{Preliminaries}, we will introduce the key ingredients that will be used to derive our main results. 
Sections \ref{Uniform}, \ref{Lower} and \ref{Concentration} are dedicated to the proofs 
of theorem \ref{IntroductionMain4}, \ref{IntroductionMain6} and \ref{IntroductionMain3} respectively.

\section{Preliminaries}\label{Preliminaries}

In this section, we introduce and prove some properties about the Kullback-Leibler 
divergence and the moment tensors. Both are key tools that play 
a central role in our analysis.

\subsection{Kullback-Leibler Divergence and Moment Tensors}\label{PreliminariesKullback}

In the MRA model, the Kullback-Leibler (KL) divergence admits an explicit formula:
\[\infdiv{P_\theta}{P_\phi} = \frac{1}{2\sigma^2} (\norm{\phi}^2 - \norm{\theta}^2) 
+ \mathbb{E}_{\xi} \bigg[\log \frac{\mathbb{E}_G[\exp(\frac{1}{\sigma^2}(\theta + \sigma\xi)^T G\theta)]}{\mathbb{E}_G[
\exp(\frac{1}{\sigma^2}(\theta + \sigma\xi)^T G\phi)]}\bigg]. \]
An important piece of machinery to analyse the KL divergence is the family of moment tensors. For any positive integer $m$, 
the \itbf{mth moment tensor} of a vector $\theta\in \R^L$ is the quantity $\mathbb{E}[(G\theta)^{\otimes m}]\in (\R^L)^{\otimes m}.$ 
If $\phi\in \R^L$ is another vector, the \itbf{mth moment difference tensor} between $\theta$ and $\phi$ is defined to be
\[\Delta_m(\theta,\phi) := \mathbb{E}[(G\theta)^{\otimes m} - (G\phi)^{\otimes m}] \in (\R^L)^{\otimes m}.\]
The relationship between the KL divergence and the moment tensors are given by the following results.
\begin{theorem}\label{PreliminariesKullback1}\cite[Lemma 8]{BRW}
Let $\theta,\phi\in \R^L$. 
Let $\upvartheta = \theta - \mathbb{E}[G\theta]$ and $\varphi = \phi - \mathbb{E}[G\phi]$. Then
\[\infdiv{P_\theta}{P_\phi} = \infdiv{P_\upvartheta}{P_\varphi} + \frac{1}{2\sigma^2}\norm{\Delta_1(\theta,\phi)}^2.\]
\end{theorem}

\begin{theorem}\label{PreliminariesKullback2}\cite[Theorem 9]{BRW}
Let $\theta,\phi\in \R^L$ satisfy $3\rho(\theta,\phi)\leq \norm{\theta} \leq \sigma$ and $\mathbb{E}[G\theta] = \mathbb{E}[G\phi] = 0$. 
For any positive integer $k$, there exist universal constants $\underline{C}$ and $\overline{C}$ such that
\[\underline{C}\sum_{m=1}^\infty 
\frac{\norm{\Delta_m(\theta,\phi)}^2}{(\sqrt{3}\sigma)^{2m} m!} \leq \infdiv{P_\theta}{P_\phi} \leq 
2\sum_{m=1}^{k-1} \frac{\norm{\Delta_m(\theta,\phi)}^2}{\sigma^{2m}m!} + \overline{C} \frac{\norm{\theta}^{2k-2} \rho(\theta,\phi)^2}{\sigma^{2k}}.\]
\end{theorem}
A few remarks are in order. Firstly, the assumption that $\mathbb{E}[G\theta] = \mathbb{E}[G\phi] = 0$ means that the first moment tensor $\Delta_1(\theta,\phi)$ vanishes. Thus the summations in the above theorem start from $m = 2$. Secondly, the result still holds if the hypothesis $3\rho(\theta,\phi) \leq \norm{\theta}$ is 
replaced by $\rho(\theta,\phi) \leq K_0 \norm{\theta}$ for any fixed positive constant $K_0\in \R_{>0}$. 
In our current setting, we let $K_0 = 3$. A full proof of this statement can be found in Appendix \ref{AppendixC}.

\subsection{Collision-Free Signals}\label{PreliminariesCollision}

As a consequence of the resolution of the Piccard's conjecture, 
collision-free signals are very well-behaved locally. This notion can be formalised in terms of the following lower bounds 
on the KL-divergence and moment tensors.

\begin{lemma}\label{PreliminariesCollision1}
\cite[Lemma 5]{Sparse}
For any pair of collision free signals $\theta,\phi\in \mc{T}$, we have
\[\norm{\Delta_2(\theta,\phi)}
\geq \sqrt{\frac{2\epsilon}{2 + \epsilon}} \cdot \frac{1}{\sqrt{L}} \cdot \sqrt{s} \cdot \rho(\theta,\phi).\]
\end{lemma}

\begin{proposition}\label{PreliminariesCollision2}\cite[Proposition 17]{Sparse}
There exists a postive constant $\epsilon_0\in \R_{>0}$ such that for all collision free signals $\theta,\phi\in \mc{T}$ satisfying $\rho(\theta,\phi) < \epsilon_0$ and for all $\sigma$ sufficiently large,
\[\infdiv{P_\theta}{P_\phi} \geq C\sigma^{-4}\cdot \norm{\Delta_2(\theta,\phi)}^2.\]
\end{proposition}

Combining the above two results, we obtain the following 
local lower bound for the KL divergence of collision free signals.

\begin{proposition}\label{PreliminariesCollision3} There 
exist positive constants $\epsilon_0,C \in \R_{>0}$ such that for any pair of collision free signals
$\theta,\phi \in \R^L$ with $\theta\in \mc{T}$ and $\rho(\theta,\phi) < \epsilon_0$ and for any 
$\sigma$ sufficiently large,
\[\infdiv{P_\theta}{P_{\phi}} \geq C\sigma^{-4} s \cdot \rho(\theta,\phi
)^2\]
The threshold $\epsilon_0$ also satisfies the following 
property: for any $\theta,\phi\in \R^L$ satisfying $\theta\in \mc{T}$ and 
$\normnorm{\theta - \phi} < \epsilon_0$, we have $\supp(\theta) = \supp(\phi)$.
\end{proposition}

\subsection{Global Lower Bound for the KL Divergence}\label{PreliminariesGlobal}

In this section, we will complement the sharp local bound for the KL divergence 
in Proposition \ref{PreliminariesCollision3} with the following global bound.

\begin{proposition}\label{PreliminariesGlobal1} Let $\theta,\phi \in \R^L$ be two collision free signals, 
with $\theta\in \mc{T}$. There exists a constant $C_s$, depending on $s$, such that
\[\infdiv{P_\theta}{P_{\phi}} \geq C_s \sigma^{-6} \rho(\theta,\phi)^2.\]
for all $\sigma$ sufficiently large.
\end{proposition}

As it turns out, the third moment tensor for collision free signals have a very simple form. 
Thus our strategy for proving Proposition \ref{PreliminariesGlobal1} will be to show that the third moment difference 
tensor is nonvanishing for two distinct collision free signals and then invoke 
Theorem \ref{PreliminariesKullback2}. 

\begin{lemma}\label{PreliminariesGlobal2} Let 
$\theta,\phi \in \R^L$ be two collision free signals lying in distinct orbits. Then $\Delta_3(\theta,\phi) \neq 0$.
\end{lemma}

\begin{proof} We will show that the orbit of any collision free signal $\theta$ 
can be recovered purely from its third moment tensor. 
The collision free property of $\theta$ implies that for any $1\leq i,j \leq L$ with $i\neq j$, we have that
\[\mathbb{E}[(G\theta)^{\otimes 3}]_{i,j,j} &= \frac{1}{L}\sum_{g\in \Z_L} \theta(i + g) \theta(j + g)\theta(j+g)
\\ &= \begin{cases}  0  & \ \text{ if } i-j \not\in \mc{D}(\theta) \\
\dfrac{1}{L}\theta(i')\theta(j')\theta(j') & \ \text{ otherwise, where } i',j'\in \text{supp}(\theta) \text{ with } i - j = i' - j'.
\end{cases}\]
In particular, the set $\mc{D}(\theta)$, 
and hence the support of $\theta$,  
can be recovered from its third moment tensor up to $\mc{R}$-orbit. Next, for each $i\in \supp(\theta)$, 
to recover $\theta(i)$, let $1\leq j \leq L$ be such that 
$i - j \in \mc{D}(\theta)$. We have that
\[\mathbb{E}[(G\theta)^{\otimes 3}]_{i,i,j} = \frac{1}{L}\theta(i)\theta(i)\theta(j) \ \ \ \ \text{ and } 
\ \ \ \ \frac{1}{L}\mathbb{E}[(G\theta)^{\otimes 3}]_{i,j,j} = \theta(i) \theta(j)\theta(j).\]
Thus \[\theta(i)^3 = \dfrac{L \cdot \mathbb{E}[(G\theta)^{\otimes 3}]_{i,i,j}^2}{\mathbb{E}[(G\theta)^{\otimes 3}]_{i,j,j}}\] 
can be obtained from the third moment tensor. The conclusion follows.
\end{proof}

\begin{lemma}\label{PreliminariesGlobal3} Let $\theta,\phi\in \R^L$. We denote $\overline{\theta} = \frac{1}{L} \sum_{i=1}^L \theta(i)$ and $\overline{\phi} = \frac{1}{L}\sum_{i=1}^L \phi(i)$. Let $\upvartheta = \theta - \E[G\theta]$ 
and $\varphi = \phi - \E[G\phi]$. 
Then
\[\Delta_3(\theta,\phi) = \Delta_3(\upvartheta,\varphi) + 3\text{Sym}\Big(\mathbbm{1} 
\otimes \big(\overline{\theta} \cdot \mathbb{E}[(G\upvartheta)^{\otimes 2}] - 
\overline{\phi} \cdot \mathbb{E}[(G\varphi)^{\otimes 2}]\big) \Big)  + (\overline{\theta}^3 - \overline{\phi}^3)\cdot \mathbbm{1}^{\otimes 3},\] 
where $\text{Sym}$ is the symmetrization operator 
which acts on order-$3$ tensors by averaging over all permutations of the indices:
\[\text{Sym}(T)_{i,j,k} := \frac{1}{6}\sum_{\sigma\in S_3} T_{\sigma(i), \sigma(j),\sigma(k)}.\]
\end{lemma}

\begin{proof} Let $\mu_\theta = \mathbb{E}[G\theta]$ and 
$\mu_\phi = \mathbb{E}[G\phi]$. Observe that
\[\mathbb{E}[(G\theta)^{\otimes 3}] &=
\mathbb{E}[(G\upvartheta + G\mu_\theta)^{\otimes 3}]  \\[1mm]
&= \mathbb{E}[(G\upvartheta + \mu_\theta)^{\otimes 3}] \\[1mm]
&= \mathbb{E}[(G\upvartheta)^{\otimes 3}] + \mathbb{E}[G\upvartheta \otimes G\upvartheta\otimes\mu_\theta]
+ \mathbb{E}[G\upvartheta \otimes\mu_\theta\otimes G\upvartheta]
+ \mathbb{E}[ \mu_\theta\otimes G\upvartheta\otimes G\upvartheta ]  \\[1mm]
&+\mathbb{E}[G\upvartheta \otimes \mu_\theta\otimes\mu_\theta]
+ \mathbb{E}[\mu_{\theta}\otimes G\upvartheta\otimes\mu_{\theta}]
+ \mathbb{E}[ \mu_{\theta}\otimes\mu_{\theta}\otimes G\upvartheta ]  
+ \mu_{\theta} \otimes \mu_{\theta}\otimes \mu_{\theta} \\[1mm]
&= \mathbb{E}[(G\upvartheta)^{\otimes 3}] + \mathbb{E}[G\upvartheta \otimes G\upvartheta\otimes\mu_\theta]
+ \mathbb{E}[G\upvartheta \otimes\mu_\theta\otimes G\upvartheta]
+ \mathbb{E}[ \mu_\theta\otimes G\upvartheta\otimes G\upvartheta ]  \\[1mm]
&+\mathbb{E}[G\upvartheta] \otimes \mu_\theta\otimes\mu_\theta
+ \mu_{\theta}\otimes \mathbb{E}[G\upvartheta]\otimes\mu_{\theta}
+ \mu_{\theta}\otimes\mu_{\theta}\otimes \mathbb{E}[G\upvartheta ]  
+ \mu_{\theta}^{\otimes 3} \\[1mm]
&= \mathbb{E}[(G\upvartheta)^{\otimes 3}] + 3\text{Sym}\big(
\mu_\theta \otimes  \mathbb{E}[(G\upvartheta)^{\otimes 2}]\big) + \mu_\theta^{\otimes 3}.\]
Since the Symmetrization operator is linear, we obtain\[\Delta_3(\theta,\phi) = \Delta_3(\upvartheta,\varphi) + 3\text{Sym}\Big(\mathbbm{1} 
\otimes \big(\overline{\theta} \cdot \mathbb{E}[(G\upvartheta)^{\otimes 2}] - 
\overline{\phi} \cdot \mathbb{E}[(G\varphi)^{\otimes 2}]\big) \Big) + (\overline{\theta}^3 - \overline{\phi}^3)\cdot \mathbbm{1}^{\otimes 3},\]
as desired.
\end{proof}

We are now ready to prove Proposition \ref{PreliminariesGlobal1}.

\begin{proof}[Proof of Proposition \ref{PreliminariesGlobal1}] Since both the KL divergence 
and the orbit distance are invariant under the action of the group $\mc{R}$, 
we assume without loss of generality that $\rho(\theta,\phi) = \norm{\theta - \phi}$. By Proposition \ref{PreliminariesCollision3} , 
it suffices to prove the above statement for the case of 
$\rho(\theta,\phi) \geq \epsilon_0.$ We will divide our argument into two cases. 

\textbf{Case 1.} $\epsilon_0 \leq \rho(\theta ,\phi) \leq 3\norm{\theta}$.

Observe that $\rho(\theta,\phi)$ is bounded above as 
$\norm{\theta}$ is (recall that we require $|\theta(i)| \leq M$ for all $1 \leq i \leq L$). Thus it suffices to show that 
\[\infdiv{P_\theta}{P_\phi} \geq \sigma^{-6} \epsilon_s\] for some small constant $\epsilon_s$ depending only 
on $s$. In what follows, we denote $\upvartheta = \theta - \E[G\theta]$ 
and $\varphi = \phi - \E[G\phi]$. Firstly, by lemma \ref{PreliminariesGlobal2},  
we have that $\Delta_3(\theta,\phi) \neq 0$. 
For each $\theta\in \mc{T}$, 
define \[B_{\theta,\epsilon_0} := \big\{\phi \in \R^L \ : \ \epsilon_0 \leq 
\rho(\theta,\phi) \leq 3\norm{\theta} \big\}.\]
Then set
\[\delta_0 := \inf_{\theta\in \mc{T}} \inf_{\phi \in B_{\theta,\epsilon_0}} \norm{\Delta_3(\theta,\phi)}.\]
Note that $\delta_0$ is a strictly positive constant 
since both $B_{\theta,\epsilon_0}$ and $\mc{T}$ are compact sets. Furthermore, $\delta_0$ is 
independent of both $\theta$ and $\phi$ (but may depend on $s$). Let $\delta\in \R_{>0}$ be a small constant, whose value may decrease from line to line, satisfying
\begin{equation}\label{PreliminariesGlobal14} \norm{3\cdot\text{Sym}\Big(\mathbbm{1} 
\otimes \big(\overline{\theta} \cdot \mathbb{E}[(G\upvartheta)^{\otimes 2}] - 
\overline{\phi} \cdot \mathbb{E}[(G\varphi)^{\otimes 2}]\big) \Big) + (\overline{\theta}^3 - \overline{\phi}^3)\cdot \mathbbm{1}^{\otimes 3}} 
\leq \frac{\delta_0}{2}
\end{equation}
for all $\theta,\phi\in \R^L$ such that $|\overline{\theta} - \overline{\phi}| \leq \delta$ and $\norm{\Delta_2(\upvartheta,\varphi)}\leq \delta$. By Theorem \ref{PreliminariesKullback1},
\[\infdiv{P_\theta}{P_\phi} = \frac{1}{2\sigma^2}\norm{\Delta_1(\theta,\phi)}^2 + \infdiv{P_\upvartheta}{P_\varphi}.\]
Thus if $|\overline{\theta}  - \overline{\phi}| > \delta$, we have that
\[\infdiv{P_\theta}{P_\phi}  \geq  \frac{1}{2\sigma^{2}}\norm{\Delta_1(\theta,\phi)}^2 
\geq \frac{\delta^2 L}{2\sigma^2} \geq \delta^2\sigma^{-6}.\]
Now suppose that $|\overline{\theta} - \overline{\phi}| \leq \delta$. If 
$\norm{\Delta_2(\upvartheta,\varphi)} > \delta$,
by \cite[Theorem 9]{BRW}, there exists an universal positive constant $C\in \R_{>0}$ 
such that
\[\infdiv{P_\theta}{P_\phi} \geq \infdiv{P_\vartheta}{P_\varphi} \geq  C \frac{\norm{\Delta_2(\upvartheta,\varphi)}^2}{\sigma^4} 
\geq \frac{C\delta^2}{\sigma^4} \geq 
\delta^2 \sigma^{-6}. \] 
Finally, suppose that we have both $|\overline{\theta} - \overline{\phi}| \leq \delta$ 
and $\norm{\Delta_2(\upvartheta,\varphi)} \leq \delta$. 
Using (\ref{PreliminariesGlobal14}) and Lemma \ref{PreliminariesGlobal3}, we obtain 
\[\norm{\Delta_3(\upvartheta,\varphi)} \geq \norm{\Delta_3(\theta,\phi)} - \frac{\delta_0}{2} \geq  \frac{\delta_0}{2}.\]
Again by Theorem \ref{PreliminariesKullback2}, we obtain
\[\infdiv{P_\theta}{P_\phi} \geq \infdiv{P_\vartheta}{P_\varphi} \geq  C \frac{\norm{\Delta_3(\upvartheta,\varphi)}^2}{\sigma^6} 
\geq \frac{C\delta^2_0}{4\sigma^6}.\]
The desired conclusion follows by choosing 
$\epsilon_s = \min \{ \frac{C\delta_0^2}{4}, \delta^2\}.$
 
\textbf{Case 2.} $\rho(\theta,\phi) > 3\norm{\theta}.$

By Lemma \ref{AppendixA1} , we in fact have a stronger bound 
$\infdiv{P_\theta}{P_\phi} \geq  C \sigma^{-4}\rho(\theta,\phi)^2$.
\end{proof}

\section{Uniform Rates of Convergence for the Maximum Likelihood Estimator}\label{Uniform}

The goal of this section is to prove Theorem \ref{IntroductionMain4}. Throughout this section, 
we fix a collision free signal $\theta \in \mc{T}$, which plays the role of 
the unknown parameter which we are trying to estimate. Let $s = |\supp(\theta)|$ and define the subspace 
\[L_\theta := \Big\{ \phi\in \R^L \ : \ \phi(i) = 0 \text{ for all } i\not\in \supp(\theta)\Big\}.\]
In this section, all derivatives will be taken with respect to the subspace $L_\theta$. In other words, 
for a (twice continuously differentiable) function $g : \R^L \to \R$ and a point $p\in \R^L$, we define the gradient and the Hessian to be
\[\nabla g(p)_i &:= \begin{cases}
\dfrac{\partial g}{\partial x_i}(p) & \text{if } i \in \supp(\theta), \\
0 & \text {otherwise};
\end{cases}  \\[2mm]
(H_g(p))_{ij} &:= \begin{cases}
\dfrac{\partial g}{\partial x_i\partial x_j}(p) & \text{if } i,j \in \supp(\theta), \\
0 & \text {otherwise}.
\end{cases} \]
Let $D : \R^L \to \R_+$ denote the map
\[D(\phi) := \infdiv{P_\theta}{P_\phi} = \mathbb{E}\bigg[\log \frac{f_\theta(X)}{f_\phi(X)}\bigg].\]
Let $H = H_{D}$ denote the Hessian of the above map at the point $\theta$ (on the subspace $L_\theta$). Since $H$ is positive semidefinite, it induces a seminorm 
$\normnorm{\cdot }_H : \R^L \to \R_+$ via
\[\norm{\phi}_H := \sqrt{\phi^T H \phi}.\]
For i.i.d samples $X_1,\cdots,X_n$ drawn according to $P_\theta$, the restricted MLE $\tilde{\theta}_n$ is a minimizer 
of 
\[D_n(\phi) := \frac{1}{n}\sum_{i=1}^n \log \frac{f_\theta}{f_\phi}(X_i)\]
on the set $\mc{T}.$

Our proof builds upon the approach in \cite[Theorem 4]{BRW}. In our proof, we will invoke the following results.


\begin{lemma}\label{Uniform1}\cite[Lemma B.6]{BRW} Let $\epsilon\in \R_{>0}$ and let $\mc{B}_\epsilon = \{\phi \in \R^L \ : \ \rho(\phi,\theta) \leq \epsilon\}$. Let $H_{D_n}$ denote the Hessian of $D_n$. There exists a constant $C$ 
such that
\[\mathbb{E}\bigg[\sup_{\phi \in \mc{B}_\epsilon} \norm{H_{D}(\phi) - H_{D_n}(\phi)}^2_{\text{op}}\bigg] \leq C\frac{\log n}{n\sigma^4}.\]
\end{lemma}

\begin{lemma}\label{Uniform2}\cite[Lemma~B.7]{BRW}
Let $\mc{L}$ be any vector subspace of $\R^L$. Suppose that there exists a positive integer $k$ and positive constants $c$ and $C$ such that for all $\theta,\phi \in \mc{L}$ satisfying 
$c^{-1} \leq \norm{\theta} \leq c$ and $\norm{\theta}\leq \sigma$, we have that
\[\infdiv{P_\theta}{P_\phi} \geq C\sigma^{-2k}\rho(\theta,\phi)^2.\]
Then the restricted MLE $\tilde{\theta}_n$ satisfies
\[\mathbb{E}[\rho(\tilde{\theta}_n,\theta)^2] \leq \tilde{C}\frac{\sigma^{4k-2}}{n}\]
for some positive constant $\tilde{C}$.
\end{lemma}

\begin{lemma}\label{Uniform3}\cite[Lemma B.15]{BRW}
There exists a positive constant $C$ 
depending only on the dimension $L$ such that for any $\theta,\phi\in \R^L$ with $\norm{\theta} \leq 1$,
\[\bigg|\infdiv{P_\theta}{P_\phi} - \frac{1}{2}\norm{\theta - \phi}^2_H \bigg| \leq C \frac{\norm{\phi - \theta}^2}{\sigma^3}.\]
\end{lemma}

\begin{proof}[Proof of Theorem \ref{IntroductionMain4}] The key idea 
behind the proof is that for $n \to +\infty$, 
the dominant term in (\ref{IntroductionMain5}) is the first term, which is controlled by the local nature of the signals. Since collision 
free signals are very well-behaved locally (Proposition 
\ref{PreliminariesCollision3}), this allows us to obtain sharper uniform bounds on the rate of convergence of the restricted MLE.

We assume without loss of generality that $\rho(\tilde{\theta}_n ,\theta) = \normnorm{\tilde{\theta}_n - \theta}.$ Furthermore, 
since $\norm{\theta}$ is bounded, we also assume $\norm{\theta} = 1$ and $\sigma \geq 1$ by replacing $\theta$ and $\sigma$ by $\theta/\norm{\theta}$ and $\sigma/\norm{\theta}$ respectively.

Let $\delta$ be a fixed positive constant and let $A$ be a positive constant whose value may change from line to line. 
Define the event $\mathcal{E} := \{\rho(\tilde{\theta}_n,\theta) \leq \epsilon\}$, where $\epsilon$ is another positive constant whose exact value will be determined later. 
We decompose
\begin{equation}\label{Uniform4}\E_\theta\big[\rho(\tilde{\theta}_n,\theta)\big] = \E_\theta\big[\rho(\tilde{\theta}_n,\theta)\mathbbm{1}_\mathcal{E}\big] +
 \E_\theta\big[\rho(\tilde{\theta}_n,\theta)\mathbbm{1}_{\mathcal{E}^c}\big]
\end{equation}
and bound each term separately. Note that the first term entails the local behaviour 
of the restricted MLE while the second term entails the global behaviour of the restricted MLE.

\textbf{Local bound:} For the first term, let $B_1$ and $B_2$ denote the constants as in Lemma \ref{Uniform3}
and Proposition \ref{PreliminariesCollision3} respectively. 
We first choose $\delta$ sufficiently small such that 
\[\frac{\delta B_1}{B_2} \leq \frac{1}{2}.\]
Let $\sigma \in \R_{>0}$ be sufficiently large such that $s\delta\sigma^{-1} < \epsilon_0$ and choose $\epsilon = s\delta\sigma^{-1}$. By Lemma \ref{Uniform3} and Proposition \ref{PreliminariesCollision3}, we have that
\begin{align*}
\Big|D(\tilde{\theta}_n) - \frac{1}{2}\normnorm{\tilde{\theta}_n-\theta}^2_H\Big| 
&\leq B_1\frac{\normnorm{\tilde{\theta}_n - \theta}^3}{\sigma^3} 
\leq \frac{B_1 s \delta\normnorm{\tilde{\theta}_n-\theta}^2}{\sigma^{4}} \leq \frac{1}{2}D(\tilde{\theta}_n).
\end{align*}
Hence 
\[\frac{1}{3}\normnorm{\tilde{\theta}_n-\theta}_H^2 \leq D(\tilde{\theta}_n) \leq \normnorm{\tilde{\theta}_n-\theta}^2_H.\]
Together with Proposition \ref{PreliminariesCollision3}, we get
\begin{align}\label{Uniform5}
\normnorm{\tilde{\theta}_n - \theta}^2_H \geq A\sigma^{-4} s\normnorm{\tilde{\theta}_n- \theta}^2.
\end{align}
The function $D_n$ satisfy $D_n(\theta) = 0$ and attains its minimum value on $\mc{T}$ at 
$\tilde{\theta}_n$. As such, we must have $D_n(\tilde{\theta}_n) \leq 0.$ 
Furthermore, since $\normnorm{\tilde{\theta}_n - \theta} < \epsilon_0$, we also have that 
$\supp(\tilde{\theta}_n) = \supp(\theta)$. Expanding $D- D_n$ as a Taylor series about $\theta$, we obtain
\begin{align*}\frac{1}{3}\normnorm{\tilde{\theta}_n -\theta}^2_H &\leq D(\tilde{\theta}_n) - 
D_n(\tilde{\theta}_n) \\ 
&-\nabla D_n(\theta)^T (\tilde{\theta}_n -\theta) + \frac{1}{2}(\tilde{\theta}_n - \theta)^T \big(H_{D}(\eta) - H_{D_n}(\eta)\big)(\tilde{\theta}_n - \theta)
\end{align*}
where $\eta\in \R^d$ is a vector on the line segment between $\theta$ and $\tilde{\theta}_n$. 
We employ the dual norm $\norm{\cdot}_H^*$ to bound the first term and the operator norm $\norm{\cdot}_{\text{op}}$ to bound the second term. This gives
\begin{align*}
\frac{1}{3}\normnorm{\tilde{\theta}_n -\theta}^2_H \leq
\norm{\nabla D_n(\theta)}^*_H \normnorm{\tilde{\theta}_n -\theta}_H + 
\dfrac{1}{2}\normnorm{\tilde{\theta}_n -\theta}^2 \sup_{\phi \in \mathcal{B}_\epsilon} \opnorm{H_{D}(\phi)- H_{D_n}(\phi)} 
\end{align*}
where $\mathcal{B}_\epsilon := \big\{\phi\in\R^d \ : \ \rho(\phi,\theta) \leq \epsilon\big\}.$ Using (\ref{Uniform5}) and dividing by $\normnorm{\tilde{\theta}_n - \theta}_H$ throughout on both sides,
\begin{align*} \sigma^{-2}\sqrt{s}\normnorm{\tilde{\theta}_n-\theta}   &\leq  A\Big(\normnorm{ \nabla D_n(\theta)}^*_H + 
\frac{\sigma^2}{\sqrt{s}} 
\normnorm{\tilde{\theta}_n-\theta}\sup_{\phi \in \mathcal{B}_\epsilon} \opnorm{H_{D}(\phi)- H_{D_n}(\phi)}\Big).
\end{align*}
Multiplying both sides by $\sigma^2/\sqrt{s}$ and applying Young's inequality to the second term on the right-hand side yields
 \begin{align*}
\normnorm{\tilde{\theta}_n-\theta}  &\leq A\bigg(\frac{\sigma^{2}}{\sqrt{s}}\norm{\nabla D_n(\theta)}^*_H + 
\frac{\sigma^7}{s}
\sup_{\phi \in \mathcal{B}_\epsilon} \opnorm{H_{D}(\phi)- H_{D_n}(\phi)}^2
+ \frac{\sigma}{s} \normnorm{\tilde{\theta}_n-\theta}^2\bigg).
\end{align*} 
Taking expectation on both sides,
\begin{align}
\begin{split}\E_\theta\big[\rho(\tilde{\theta}_n,\theta)\mathbbm{1}_\mathcal{E}\big] &\leq  
\frac{A\sigma^2}{\sqrt{s}}\cdot\mathbb{E}_\theta\big[\norm{\nabla  D_n(\theta)}^*_H\big] + 
\frac{A\sigma^{7}}{s}\cdot\mathbb{E}_\theta\bigg[ 
\sup_{\phi \in \mathcal{B}_\epsilon} \opnorm{H_{D}(\phi)- H_{D_n}(\phi)}^2\bigg]
\\ &+ \frac{A\sigma}{s}\cdot \mathbb{E}_\theta\big[\rho(\tilde{\theta}_n,\theta)^2\big]. \label{Uniform6}
\end{split}
\end{align}
This concludes the local analysis of the restricted MLE.

\textbf{Global bound:} We first address the second term in (\ref{Uniform4}). Since $\rho(\tilde{\theta}_n,\theta) \geq  s\delta \sigma^{-1}$ on $\mc{E}^c$, we obtain 
\[\mathbb{E}_\theta\big[\rho(\tilde{\theta}_n,\theta)\mathbbm{1}_{\mathcal{E}^c}\big] \leq \frac{A\sigma}{s}\cdot 
\mathbb{E}_\theta\big[\rho(\tilde{\theta}_n,\theta)^2\big] \]which can be combined with the third term in (\ref{Uniform6}). This 
gives
\begin{align}
\begin{split}\label{Uniform7}
\E_\theta\big[\rho(\tilde{\theta}_n,\theta)\big] &\leq  
\frac{A\sigma^2}{\sqrt{s}}\cdot\mathbb{E}_\theta\big[\norm{\nabla  D_n(\theta)}^*_H\big] + 
\frac{A\sigma^{7}}{s}\cdot\mathbb{E}_\theta\bigg[ 
\sup_{\phi \in \mathcal{B}_\epsilon} \opnorm{H_{D}(\phi)- H_{D_n}(\phi)}^2\bigg]
\\ &+ \frac{A\sigma}{s}\cdot \mathbb{E}_\theta\big[\rho(\tilde{\theta}_n,\theta)^2\big].
\end{split}
\end{align}
We will establish an upper bound for each of the three terms in the above equation separately.

For the first term, Bartlett's identities state that for each $1\leq i\leq n$,
\begin{align*}\mathbb{E}_\theta\big[\nabla \log f_\phi(X_i)|_{\phi = \theta}\big] &= 0  \\
\mathbb{E}_\theta\big[(\nabla \log f_\phi(X_i)|_{\phi = \theta})
(\nabla \log f_\phi(X_i)|_{\phi = \theta})^T\big] &= H_D(\theta) = H .
\end{align*}
Since $\nabla D_n(\theta) = \displaystyle -\frac{1}{n}\sum_{i=1}^n \nabla \log f_\phi(X_i)\big|_{\phi = \theta}$ and the $X_i$'s are independent, we get
\[\mathbb{E}_\theta\big[\nabla D_n(\theta) \nabla D_n(\theta)^T\big] = \frac{1}{n} H.\]
In particular, the nullspace of $H$ is contained in the nullspace of $\nabla D_n(\theta) \nabla D_n(\theta)^T$ almost surely since any vector $v$ lying in the nullspace of $H$ satisfies
\[0 \leq \mathbb{E}_\theta\big[v^T
\nabla D_n(\theta) \nabla D_n(\theta)^T v\big] = \frac{1}{n} v^T H v = 0.\]
Since the row space of a symmetric matrix is the orthogonal complement of its nullspace, this means that the row space of $\nabla D_n(\theta) \nabla D_n(\theta)^T$ (and hence the row space of $\nabla D_n(\theta))$ is contained in the row space of $H$ almost surely. As a result,
\begin{align}\mathbb{E}_\theta\big[\norm{\nabla D_n(\theta)}^*_H\big] 
&= \mathbb{E}_\theta\Big[\sqrt{\nabla D_n(\theta)^T H^\dagger \nabla D_n(\theta)}\Big] \nonumber\\ &\leq 
\sqrt{\mathbb{E}_\theta\big[\text{Tr}\big(\nabla D_n(\theta)^T H^\dagger \nabla D_n(\theta)\big)\big]} \nonumber\\ 
&= \sqrt{\frac{1}{n} \text{Tr}(H^\dagger H)} \leq \sqrt{\frac{s}{n}} \label{Uniform8},
\end{align}
where $H^\dagger$ denotes the Moore-Penrose inverse of $H$. For the second term, we invoke Lemma \ref{Uniform1} to get
\begin{equation}\label{Uniform9}\mathbb{E}_\theta\bigg[ 
\sup_{\phi \in \mathcal{B}_\epsilon} \opnorm{H_{D}(\phi)- H_{D_n}(\phi)}^2\bigg] \leq 
A\frac{\log n}{n\sigma^4}.
\end{equation}
Finally, for the third term, we invoke Lemma \ref{Uniform2}
with the weaker global bound in Proposition \ref{PreliminariesGlobal1} (i.e. $k= 3$). We obtain
\begin{equation}\label{Uniform10}
\mathbb{E}_\theta\big[\rho(\tilde{\theta}_n,\theta)^2\big] \leq A_s \frac{\sigma^{10}}{n}
\end{equation}
for some positive constant $A_s$ depending on $s$.
Putting (\ref{Uniform8}),\  (\ref{Uniform9}) and (\ref{Uniform10}) back into (\ref{Uniform7}), we have
\[\mathbb{E}\big[\rho(\tilde{\theta}_n,\theta)\big] \leq A\frac{\sigma^2}{\sqrt{n}} + A\frac{\sigma^3 \log n}{sn} + 
A_s \frac{\sigma^{11}}{sn}.\]
The conclusion follows.
\end{proof}


\section{Lower Bound for Collision Free Signals}\label{Lower}

In this section, we prove Theorem \ref{IntroductionMain6}.

\begin{proof} For convenience of notation, we assume that $s$ is even. The proof for the case in which 
$s$ is odd is conceptually similar but notationally more complicated. By Le Cam's two point argument, for any $\phi,\psi\in \mc{T}_s$, we have that
\[\inf_{\hat{\theta}_n} \sup_{\theta\in \mc{T}_s} \mathbb{E}_{\theta} \big[\rho(\hat{\theta}_n,\theta)\big] \geq \rho(\phi,\psi)\cdot\left(\frac{2-\sqrt{2n\infdiv{P_\phi}{P_\psi}}}{8}\right)\]
where the infimum is taken over all estimators $\hat{\theta}_n$ on $n$ samples. Thus it suffices to construct two signals $\theta_n,\phi\in \mc{T}_s$ such that 
\[\rho(\phi,\theta_n) \geq M \min\bigg\{\frac{\sigma^{2}}{\sqrt{n}},\ 1 \bigg\} \]
for some universal constant $M\in \R_{>0}$ but $\infdiv{P_\phi^n}{P^n_{\theta_n}} = n\infdiv{P_{\phi}}{P_{\theta_n}} \leq 1/2$. 
Let $\{i_1,i_2,\cdots,i_s\}\subseteq \{1,\cdots,L\}$ be such that each element in the multiset 
\[\big\{i_j - i_k \ : \ 1\leq j,k \leq s \text{ with } j \neq k\big\}\]
appears with multiplicity exactly $1$. Define
\[\phi_j &:= \begin{cases}
(-1)^k\dfrac{1}{\sqrt{s}} & \text{ if } j = i_k \text{ for } 1 \leq k \leq s \\
0 & \text{ otherwise} 
\end{cases} \\[2mm]
(\theta_n)_j &:= \begin{cases}
(-1)^k\dfrac{1}{\sqrt{s}} & \text{ if } j = i_k \text{ for } 1 \leq k \leq s - 2 \\[3mm]
-\dfrac{1}{\sqrt{s}} - \delta & \text{ if } j = i_k \text{ for } k = s-1 \\[3mm]
\dfrac{1}{\sqrt{s}} + \delta & \text{ if } j = i_k \text{ for } k = s \\[3mm]
0 & \text{ otherwise} 
\end{cases}\]
where $\displaystyle\delta = c\min\bigg\{\frac{\sigma^2}{\sqrt{n}}, 1\bigg\}$ 
for some universal constant $c$ to be specified.

Note that $\mathbb{E}[G\phi ] = \mathbb{E}[G\theta] = 0$ and $\rho(\phi,\theta_n) = 2\delta \leq 1/3$ for $c < 1/6$. Since both $\phi$ and $\theta_n$ 
are mean zero signals, with $\norm{\phi} = 1$, we can invoke Theorem \ref{PreliminariesKullback2} (with $k = 2$) to obtain
\[\infdiv{P_\phi}{P_{\theta_n}} \leq \overline{C} \frac{\rho(\theta,\phi)^2}{\sigma^{4}} \leq \frac{1}{2n}\]
by choosing $c < \dfrac{1}{2\overline{C}}$.
\end{proof}


\section{Concentration of the Maximum Likelihood Estimator}\label{Concentration}

In this section, we fix a collision free signal $\theta\in \mc{T}$ that plays the role of 
the true signal which we wish to approximate. The goal of this section is to prove 
Theorem \ref{IntroductionMain3}.

Recall that for i.i.d samples 
$X_1,\cdots,X_n$ drawn according to $P_\theta$, the restricted MLE $\tilde{\theta}_n$ is the minimizer of the negative log-likehood
\begin{align*}
R_n(\phi) :&= -\frac{1}{n}\sum_{i=1}^n \log f_\phi (X_i) \\
&= \frac{L}{2}\log 2\pi \sigma^2 + \frac{1}{n} \sum_{i=1}^n 
\bigg(\frac{\norm{G\theta + \sigma\xi_i}^2 + \norm{\phi}^2}{2\sigma^2} \bigg) 
- \log \mathbb{E}_G\bigg[\exp\bigg(\frac{(G\theta + \sigma \xi_i)^T G\phi}{\sigma^2} \bigg) \bigg]\bigg). 
\end{align*}
on $\mc{T}$. Since $G\theta + \sigma \xi$ is equal in distribution to $G(\theta + \sigma \xi)$, we obtain
\begin{align}
R_n(\phi) \stackrel{d}{=} 
\frac{L}{2}\log 2\pi \sigma^2 + \frac{1}{n} \sum_{i=1}^n 
\bigg(\frac{\norm{\theta + \sigma\xi_i}^2 + \norm{\phi}^2}{2\sigma^2} \bigg) 
- \log \mathbb{E}_G\bigg[\exp\bigg(\frac{(\theta + \sigma \xi_i)^T G\phi}{\sigma^2} \bigg) \bigg]\bigg) \label{Concentration1}.
\end{align}
In particular, $R_n(\phi)$ does not depend on the randomness of the group action. 
Thus we let $R(\phi) = \mathbb{E}[R_n(\phi)]$ denote its mean with respect to the Gaussian noise $\xi_1,\cdots,\xi_n$.

\begin{lemma}\label{Concentration2} Let $K,\delta\in \R_{>0}$ be fixed constants. There exists a $\delta$-net $N_\delta$ (with 
respect to the Euclidean norm on $\R^L$) for the subset 
\[S := \Big\{\phi\in \R^L \ : \ |\text{supp}(\phi)| \leq s \text{ and } \norm{\phi} \leq K \Big\}\]
satisfying $|N_\delta| \leq L^s\bigg(\dfrac{3K}{\delta}\bigg)^s$.
\end{lemma}

\begin{proof} For each tuple of indices $(i_1,\cdots,i_s)$ with $1 \leq i_1 < i_2 < \cdots < i_s \leq L$, define the set 
\[S_{i_1,\cdots,i_s} := 
\Big\{\phi\in \R^L \ : \ \text{supp}(\phi) \subseteq \{i_1,\cdots,i_s\} \text{ and } \norm{\phi} \leq K \Big\}.\]
By viewing $S_{i_1,\cdots,i_s}$ as a subset of the $s$-dimensional linear subspace
\[L_{i_1,\cdots,i_s} = 
\Big\{\phi\in \R^L \ : \ \text{supp}(\phi) \subseteq \{i_1,\cdots,i_s\} \Big\},\]
we see that there exists a $\delta$-net $N^{(i_1,\cdots,i_s)}_\delta$ for $S_{i_1,\cdots,i_s}$ 
satisfying $|N^{(i_1,\cdots,i_s)}_\delta| 
\leq (3K/\delta)^s.$ Define $N_\delta$ to be the union of 
the $N_\delta^{(i_1,\cdots,i_s)}$'s.
Since there are ${L}\choose{s}$ possible choices for $i_1,\cdots, i_s$, we have that
\[|N_\delta| \leq \begin{pmatrix} L \\ s \end{pmatrix} 
\bigg(\frac{3K}{\delta}\bigg)^s \leq L^s\bigg(\frac{3K}{\delta}\bigg)^s.\]
\end{proof}

\begin{lemma}\label{Concentration3} 
There exists a universal constant $c\in \R_{>0}$ and a constant $C_{s}\in \R_{>0}$ depending on $s$ such that 
for any $K,t\in \R_{>0}$, 
\[P\bigg(\sup_{\phi \in \mc{T} : \norm{\phi} \leq K} |&R_n(\phi) - R(\phi)| > 4t \bigg) \\[1mm]
&\leq  C_{s} \sigma^{-s} K^s (t^{-s} + t^{-s/2}) \exp\bigg(-  \frac{cn\sigma^2 t^2}{K^2}\bigg)
+  C_{s}\exp\big(-cn \min\{t, t^2\}\big).\]
\end{lemma}

\begin{proof} In what follows, let $c$ and $C_s$ be constants whose value may change from line to line. 
Using (\ref{Concentration1}), for any $\phi \in \R^{d}$, we may write
\[R_n(\phi) = \RN{1} - \RN{2}(\phi) + \text{const}(\phi)\]
where
\[\RN{1} &= \frac{1}{n}\sum_{i=1}^n \frac{1}{2\sigma^2} \normnorm{\theta + \sigma\xi_{i}}^2 \\[1mm]
\RN{2}(\phi) &= \frac{1}{n}\sum_{i=1}^n f(\xi_i,\phi) := 
\frac{1}{n}\sum_{i=1}^n \log \mathbb{E}_G\bigg[  \exp\bigg(\frac{(\theta 
+ \sigma\xi_i)^T G \phi }{ \sigma^2}\bigg) \bigg]  \]
and $\text{const}(\phi)$ is a term not depending on the randomness of 
the $\xi_i$'s. We analyse the concentration of $\RN{1}$ 
and $\RN{2}(\phi)$ separately. Define the event $\mathcal{E} = \{ | \RN{1} - \mathbb{E}[\RN{1}]| < 2t \}$. 
First observe that for each $1\leq i \leq n$,
\[\frac{1}{2\sigma^2} \normnorm{\theta + \sigma\xi_i}^2 
= \frac{1}{2\sigma^2}\big(\norm{\theta}^2 + 2\sigma\theta^T \xi_i 
+ \sigma^2\normnorm{\xi_i}^2\big) .\]
The first term is deterministic while the distributions of the second and third terms are given by
\[\frac{1}{\sigma n} \sum_{i=1}^n \theta^T \xi_i 
\sim \mc{N}\bigg(0, \frac{1}{\sigma^2n} \norm{\theta}^2\bigg)\ \ \ \ \ \ \ 
\text{ and } \ \ \ \ \ \ \
\sum_{i=1}^n \normnorm{\xi_{i}}^2 \sim \chi_{nL}^2.\]
By Gaussian tail bounds, we have that
\[P\Bigg( \bigg| 
\frac{1}{\sigma n} \sum_{i=1}^n \theta^T\xi_i  \bigg| 
\geq t \Bigg)  \leq 2\exp\bigg( -\frac{n\sigma^2 t^2}{2 \normnorm{\theta}^2}\bigg).\]
For the chi-square random variable, we will invoke the following well-known chi-square tail bound. 
\begin{theorem}\label{Concentration4}
\cite[Lemma 1]{Chi-square}. 
Let $X \sim \chi^2_k$ be a chi-squared distributed random variable with $k$ degrees of freedom. Then for any $x\in \R$,
\begin{align*}P(X - k \geq 2\sqrt{kx} + 2x) &\leq \exp(-x) \\
P(k - X \geq 2\sqrt{kx}) &\leq \exp(-x).
\end{align*}
\end{theorem}

Observe that
\[P\bigg( \frac{1}{2n}\sum_{i=1}^n \normnorm{\xi_i}^2 - 
\frac{L}{2} \geq t\bigg) &=
P\bigg(\sum_{i=1}^n \normnorm{\xi_i}^2 - 
Ln \geq 2nt\bigg).\]
Let $x\in\R_{>0}$ be such that $\sqrt{nL x} + x = nt.$ If $x \geq \dfrac{nt}{2}$, then 
$n\sqrt{2Lt} + nt \leq 2nt$ and so
\[P\bigg(\sum_{i=1}^n \normnorm{\xi_i}^2 - 
Ln \geq 2nt\bigg) \leq 
P\bigg(\sum_{i=1}^n \normnorm{\xi_i}^2 - 
Ln \geq  n\sqrt{2Lt} + nt\bigg) \leq \exp\bigg( - \frac{nt}{2}\bigg).\]
On the other hand, if $\sqrt{nLx} \geq \dfrac{nt}{2}$, then $x \geq \dfrac{nt^2}{4L}$ and so 
$nt + \dfrac{nt^2}{2L} \leq 2nt$. Thus
\[P\bigg(\sum_{i=1}^n \normnorm{\xi_i}^2 - 
Ln \geq 2nt\bigg) \leq 
P\bigg(\sum_{i=1}^n \normnorm{\xi_i}^2 - 
Ln \geq  nt + \frac{nt^2}{2L} \bigg) \leq \exp\bigg( - \frac{nt^2}{4L}\bigg).\]
Next, we have that
\[ 
P\bigg(\frac{L}{2} - \frac{1}{2n}\sum_{i=1}^n \normnorm{\xi_i}^2  \geq t\bigg)  
&=  P\bigg(Ln - \sum_{i=1}^n \normnorm{\xi_i}^2 \geq 2nt\bigg)  
\leq \exp\bigg( - \frac{nt^2}{L}\bigg).\]
To summarise, we have the following bound
\[P\Bigg( \bigg|\frac{1}{2n}\sum_{i=1}^n \norm{\xi_i}^2 - 
\frac{L}{2} \bigg|\geq t\Bigg)  \leq 2\exp\big(- cn \min\{t, t^2\}\big)\]
This implies that
\begin{align}
P(\mathcal{E}^c) &\leq  2\exp\bigg( -\frac{n\sigma^2 t^2}{2 \norm{\theta}^2}\bigg) 
+2\exp\big(- cn \min\{t, t^2\}\big) \notag \\[1mm]
&\leq C_s\exp\big(-c n \min\{t, t^2\}\big) \label{Concentration5}\end{align}
where we have used the fact that $\norm{\theta}\leq \sigma$ for the last inequality.

On the event $\mathcal{E}$, we have $|R_n(\phi) - R(\phi)| \leq 2t + |\RN{2}(\phi) - \mathbb{E}[\RN{2}(\phi)]|$ as well as
\begin{equation}\label{Concentration6}
\frac{1}{n}\sum_{i=1}^n \normnorm{\theta + \sigma\xi_i}^2   
\leq \mathbb{E}\big[\normnorm{\theta + \sigma\xi_i}^2] + 
4t\sigma^2
= \norm{\theta}^2 + (L + 4t)\sigma^2 \leq 
(2L + 4t)\sigma^2.
\end{equation}
We now control the concentration of $\RN{2}(\phi)$. 
Differentiating $f(\xi,\phi)$ 
with respect to $\xi$, we obtain
\[\norm{\nabla_\xi f(\xi,\phi)} = 
\norm{\frac{\mathbb{E}_G\big[\exp(\frac{(\theta + \sigma\xi)^T G\phi}{\sigma^2}) \cdot \frac{1}{\sigma} G\phi\big]}{
\mathbb{E}_G\big[\exp(\frac{(\theta + \sigma\xi)^T G\phi }{\sigma^2})\big]}} 
\leq   \frac{\mathbb{E}_G\big[\exp(\frac{(\theta + \sigma\xi)^T G\phi }{\sigma^2}) \cdot \frac{1}{\sigma} \norm{\phi}\big]}{
\mathbb{E}_G\big[\exp(\frac{(\theta + \sigma\xi)^T G\phi }{\sigma^2})\big]}
= \frac{\norm{\phi}}{\sigma}.\]
Thus $f(\xi,\phi)$ is $\frac{\norm{\phi}}{\sigma}$-Lipschitz in $\xi$. By 
the following 
Gaussian concentration result:

\begin{theorem}\label{Concentration8}\cite[Theorem 5.5]{Concentration} 
Let $X \sim \mc{N}(0,I_n)$ be a standard 
Gaussian vector. Let $L \in \R_{>0}$ and let $f : \R^n \to \R$ be an $L$-lipschitz function. Then for all $\lambda\in \R$,
\[\log \mathbb{E}[e^{\lambda(f(X)- \mathbb{E}[f(X)])} \leq \frac{\lambda^2}{2}L^2.\]
\end{theorem}

for a fixed $\phi\in \R^{d}$, the random variable $f(\xi,\phi) - 
\mathbb{E}[f(\xi,\phi)]$ is $\frac{\norm{\phi}^2}{\sigma^2}$-subgaussian. 
Hoeffding's inequality then yields
\[P\big( |\RN{2}(\phi) - \mathbb{E}[\RN{2}(\phi)] | > t\big) 
&= P\bigg(\frac{1}{n}\bigg|\sum_{i=1}^n f(\xi_i , \phi) - \mathbb{E}[f(\xi_i,\phi)] \bigg|\geq t\bigg) \\[1mm]
&= P\bigg(\bigg|\sum_{i=1}^n f(\xi_i , \phi) - \mathbb{E}[f(\xi_i,\phi)] \bigg|\geq n t\bigg) \\[1mm]
&\leq 2\exp\bigg( - \frac{n\sigma^2 t^2}{2\norm{\phi}^2} \bigg).
\] 
Now set $\delta = \dfrac{t\sigma}{4\sqrt{L + t}}$ and let 
$N_\delta$ be a $\delta$-net of $\{\theta \in\mc{T} \ : \ \norm{\theta} \leq M\}$ having cardinality
\[|N_\delta| \leq L^s \bigg(\frac{12K\sqrt{L+t}}{t\sigma}\bigg)^s. \]
Next, the gradient of $\mathbb{E}_{\xi}[f(\xi,\phi)]$ when differentiating with respect to $\phi$ is given by
\[\norm{\nabla_\phi \mathbb{E}_{\xi}[ f(\xi,\phi)]} 
&= \norm{\mathbb{E}_\xi\Bigg[
\frac{\mathbb{E}_G\big[\exp(\frac{(G^T(\theta + \sigma\xi))^T\phi }{\sigma^2}) \cdot \frac{G^T(\theta + \sigma\xi)}{\sigma^2}\big]}{
\mathbb{E}_G\big[\exp(\frac{(G^T(\theta + \sigma\xi))^T\phi }{\sigma^2})\big]}\Bigg]} \\[2mm] 
&\leq \mathbb{E}_\xi\Bigg[
\frac{\mathbb{E}_G\big[\exp(\frac{(G^T(\theta + \sigma\xi))^T\phi }{\sigma^2}) \cdot \normnorm{\frac{\theta + \sigma\xi}{\sigma^2}}\big]}{
\mathbb{E}_G\big[\exp(\frac{(G^T(\theta + \sigma\xi))^T \phi }{\sigma^2})\big]}\Bigg] 
= \frac{1}{\sigma^2}\mathbb{E}_\xi[\norm{\theta + \sigma\xi}].  \]
This can be upper bounded by
\[\frac{1}{\sigma^2}\mathbb{E}_\xi[\norm{\theta + \sigma\xi}] \leq 
\frac{\norm{\theta} + \sigma \mathbb{E}[\norm{\xi}]}{\sigma^2} \leq 
\frac{\sigma + \sigma \sqrt{\mathbb{E}[\normnorm{\xi}^2]}}{\sigma^2} = \frac{2\sqrt{L}}{\sigma}.\]
On the other hand, on the event $\mc{E}$, for a fixed $\xi\in \R^{d}$, 
the gradient of $\RN{2}(\phi)$ is given by
\[\norm{\nabla_\phi\RN{2}(\phi)} 
&= \norm{\frac{1}{n}\sum_{i=1}^n \Bigg[
\frac{\mathbb{E}_G\big[\exp(\frac{(G^T(\theta + \sigma\xi_i))^T\phi }{\sigma^2}) \cdot \frac{G^T(\theta + \sigma\xi_i)}{\sigma^2}\big]}{
\mathbb{E}_G\big[\exp(\frac{(G^T(\theta + \sigma\xi_i))^T \phi }{\sigma^2})\big]}\Bigg]} \\[2mm] 
&\leq \frac{1}{n}\sum_{i=1}^n \Bigg[
\frac{\mathbb{E}_G\big[\exp(\frac{(G^T(\theta + \sigma\xi_i))^T\phi }{\sigma^2}) \cdot \normnorm{\frac{\theta + \sigma\xi_i}{\sigma^2}}\big]}{
\mathbb{E}_G\big[\exp(\frac{(G^T(\theta + \sigma\xi_i))^T \phi }{\sigma^2})\big]}\Bigg] 
= \frac{1}{\sigma^2} \cdot \frac{1}{n} \sum_{i=1}^n \norm{\theta + \sigma\xi_i}.  \]
This can be upper bounded by
\[\frac{1}{\sigma^2}\cdot \frac{1}{n}\sum_{i=1}^n\normnorm{
\theta + \sigma\xi_i} 
\leq  \frac{1}{\sigma^2} \sqrt{\frac{1}{n}\sum_{i=1}^n\normnorm{
\theta + \sigma\xi_i}^2} \leq \frac{\sqrt{2L + 4t}}{\sigma}\]
where the last inequality follows from $(\ref{Concentration6})$. 
This means that on the event $\mc{E}$, the function 
$\RN{2}(\phi) - \mathbb{E}[\RN{2}(\phi)]$ (for each fixed $\xi$)
is Lipschitz in $\phi$, with Lipschitz constant at most
\[\frac{2\sqrt{L}}{\sigma} + \frac{\sqrt{2L+4t}}{\sigma} \leq 
\frac{4\sqrt{L + t}}{\sigma} = \frac{t}{\delta}.\]
It then follows that 
\[P\bigg(\sup_{\phi\in\mc{T} : \norm{\phi} \leq K } | R_N(\phi) - R(\phi)| > 4t \text{ and } \mc{E} \bigg) 
&\leq P\bigg(\sup_{\phi\in\mc{T} : \norm{\phi} \leq K } |\RN{2}(\phi) - \mathbb{E}[\RN{2}(\phi)]| > 
2t \text{ and } \mc{E} \bigg) \\
&\leq 
P\bigg(\sup_{\phi \in N_\delta } |\RN{2}(\phi) - \mathbb{E}[\RN{2}(\phi)]| > 
t\bigg) \\
&\leq 2|N_\delta|\exp\bigg( - \frac{n\sigma^2 t^2}{2K^2} \bigg) \\
&\leq C_s\bigg(\frac{K\sqrt{L+t}}{t\sigma}\bigg)^s\exp\bigg(-  \frac{cn\sigma^2 t^2}{K^2}\bigg) \\
&\leq C_s\sigma^{-s}K^s\bigg(\frac{\sqrt{1+t}}{t}\bigg)^s\exp\bigg(-  \frac{cn\sigma^2 t^2}{K^2}\bigg) \\
&\leq C_{s} \sigma^{-s}K^s (t^{-s} + t^{-s/2}) \exp\bigg(-  \frac{cn\sigma^2 t^2}{K^2}\bigg).\]
Combining this with (\ref{Concentration5}) gives the desired result.
\end{proof}

\begin{proof}[Proof of Theorem \ref{IntroductionMain3}] In what follows, let $C_s$ and $\tilde{C}_s$ be constants depending on $s$ 
whose value may change from line to line. We employ a slicing argument.
For the given value of $\delta$ and each integer $k\in\Z_{\geq 1}$, define
\[\Gamma_k = \Big\{ \phi \in \R^L : k\delta 
\leq \rho(\theta,\phi) < (k+1)\delta \Big\}.\]
Then for all $\phi\in \Gamma_k$, we have that
\[\infdiv{P_\theta}{P_\phi} \geq C_s \sigma^{-6}\rho(\theta,\phi)^2 
\geq C_s \sigma^{-6} k^2 \delta^2\]
by Proposition \ref{PreliminariesGlobal1}. Let $t_k = C_s \sigma^{-6}k^2 \delta^2$. By shrinking the constant $C_s$, we have that
\[\infdiv{P_\theta}{P_\phi} \geq 10 t_k.\]
Note that $\max\{2\sqrt{L}M , k\delta\} \leq 2\sqrt{L}M k$ 
since $\delta \leq 2\sqrt{L}M$. Invoking Lemma \ref{Concentration3} with $t = t_k$ and $K = \max\{ 2\sqrt{L} M, k\delta\}$ , we have that
\[&\ \ \ \ P\bigg(\sup_{\phi \in \mc{T} : \norm{\phi} \leq \max\{2\sqrt{L} M,k\delta\} } 
| R_n(\phi) - R(\phi)| > 4t_k \bigg) \\[2mm]
&\leq C_{s}(\sigma^{5s}k^{-s}\delta^{-2s} + 
\sigma^{2s} \delta^{-s}) \exp\bigg(-  \frac{\tilde{C}_s n k^2 \delta^4}{\sigma^{10}}\bigg) 
+ 
C_{s} \exp\bigg(- \tilde{C}_s n \min \bigg\{\frac{k^2\delta^2}{\sigma^6} , \frac{k^4\delta^4}{\sigma^{12}}\bigg\}\bigg) \\[2mm]
&\leq C_{s}\sigma^{5s}\delta^{-2s} \exp\bigg(-  \frac{\tilde{C}_s n k^2 \delta^4}{\sigma^{10}}\bigg) + 
C_{s} \exp\bigg(-\frac{ \tilde{C}_s  nk^2\delta^4}{\sigma^{12}}\bigg).
\]
Note that $\infdiv{P_\theta}{P_\phi} = R(\phi) - R(\theta) \geq 10t_k$. Thus in the event in which we have both 
$|R_n(\theta) - R(\theta)| \leq 4t_k$ and 
$|R_n(\phi) - R(\phi)| \leq 4t_k$, since $R_n(\phi) - R_n(\theta) \geq 2t_k \geq 0$, 
we see that $\phi$ is not the restricted MLE (recall that the restricted MLE $\tilde{\theta}_n$ 
minimises $R_n(\tilde{\theta}_n)$ on $\mc{T}$). Thus
\[
P(\tilde{\theta}_n \in \Gamma_k ) &\leq C_{s}\sigma^{5s}\delta^{-2s} \exp\bigg(-  \frac{\tilde{C}_s n k^2 \delta^4}{\sigma^{10}}\bigg) + 
C_{s} \exp\bigg(-\frac{ \tilde{C}_s  nk^2\delta^4}{\sigma^{12}}\bigg) \\[1mm]
&\leq C_{s}\sigma^{5s}\delta^{-2s} \exp\bigg(-  \frac{\tilde{C}_s n k^2 \delta^4}{\sigma^{12}}\bigg) + 
C_{s} \exp\bigg(-\frac{ \tilde{C}_s  nk^2\delta^4}{\sigma^{12}}\bigg)   \\[1mm]
&\leq C_s\sigma^{5s}\delta^{-2s} \exp\bigg(-  \frac{\tilde{C}_s n k^2 \delta^4}{\sigma^{12}}\bigg)\]
where the last inequality uses the fact that $\sigma^{5s} \delta^{-2s}$ is bounded below by a constant. 
Summing over $k\in \Z_{\geq 1}$ gives
\[P(\rho(\tilde{\theta}_n,\theta) \geq \delta) &\leq 
\sum_{k=1}^\infty 
P(\tilde{\theta}_n \in \Gamma_k) \\ 
&\leq C_s\sigma^{5s}\delta^{-2s} \sum_{k=1}^\infty \exp\bigg(-  \frac{\tilde{C}_s n k^2 \delta^4}{\sigma^{12}}\bigg).\]Now 
for $n \geq \widetilde{C}_s^{-1} \delta^{-4} \sigma^{12},$ the infinite summation is bounded by the geometric series 
\[\exp\bigg( - \frac{\tilde{C}_s n \delta^4}{\sigma^{12}}\bigg)\sum_{k=0}^\infty \exp(- k),\]
which in turn is dominated by the first term. Thus
\[P(\rho(\tilde{\theta}_n,\theta) \geq \delta) \leq C_s \sigma^{5s}\delta^{-2s} \exp\bigg(-\frac{\tilde{C}_s n\delta^4}{\sigma^{12}}\bigg)\]
as desired.
\end{proof}

\section{Acknowledgements}
SG was supported in part by the MOE grants  R-146-000-250-133, R-146-000-312-114 and MOE-T2EP20121-0013. 

SSM was partially supported by an INSPIRE research grant (DST/INSPIRE/04/2018/002193) from the Department of Science and Technology, Government of India and a Start-Up Grant from Indian Statistical Institute, Kolkata.
\pagebreak

\addcontentsline{toc}{section}{References}

\bibliographystyle{plain}
\bibliography{Bibliography}

\pagebreak

\appendix

\addcontentsline{toc}{section}{Appendices}
\renewcommand{\thesubsection}{\Alph{subsection}}

\subsection{Appendix A}\label{AppendixA}

\begin{lemma}\label{AppendixA1} 
Let $\theta,\phi\in \R^L$ be two vectors and 
suppose that $\rho(\theta ,\phi) \geq 3\norm{\theta}$. Then 
there exists a constant $C_L$, depending on $L$, 
such that $\infdiv{P_\theta}{P_\phi} \geq C_L \sigma^{-4} \rho(\theta,\phi)^2.$
\end{lemma}

\begin{proof} By replacing $\theta, \phi$ and $\sigma$ with 
$\theta/\norm{\theta}, \phi/\norm{\theta}$ and $\sigma/\norm{\theta}$ 
respectively, we may assume that $\norm{\theta} = 1$ and $\sigma \geq 1$. 
We further assume without loss of generality that $\rho(\theta,\phi) = \norm{\theta - \phi}$. 
Now observe that 
\[ \big|\norm{\theta - \phi} - \norm{\phi}\big| \leq \norm{\theta} \leq \frac{1}{3} \norm{\theta - \phi},\]
where the left-hand side follows from two different applications of the triangle inequality. 
Thus $\norm{\phi} \approx \norm{\theta  - \phi}$ so it suffices to show that 
there exists a positive constant $C_L$ such that $\infdiv{P_\theta}{P_\phi} \geq C_L\sigma^{-4}\norm{\phi}^2$ whenever 
$\norm{\phi} \geq 2$. We will directly use the Donsker-Varadhan formula
\[\infdiv{P_\theta}{P_\phi} = \sup_f \bigg\{\underset{X\sim P_\theta}{\mathbb{E}}\big[f(X)\big] - 
\log\Big(\underset{X\sim P_\phi}{\mathbb{E}}
\big[e^{f(X)}\big]\Big)\bigg\}\]
with $f(x) := -\lambda\norm{x}^2$, where $\lambda \geq 0$ is a constant to be specified. For the first term, we have 
\[\underset{X\sim P_\theta}{\mathbb{E}}\Big[-\lambda\norm{X}^2\Big] &= 
\underset{X\sim P_\theta}{\mathbb{E}}\Big[-\lambda\norm{G\theta + \sigma\xi}^2\Big]  \\[1mm]
&= - \lambda \underset{X\sim P_\theta}{\mathbb{E}}\Big[ \norm{G\theta}^2 + 2\sigma(G\theta)^T\xi +  \sigma^2\norm{\xi}^2\Big] \\[1mm]
&= -\lambda(1+\sigma^2L). \]
For the second term, note that for $X = G\phi + \sigma\xi$, conditioned on $G$, 
the random variable $\norm{\frac{1}{\sigma}X}^2$ has a non-central $\chi^2$-distribution 
with $L$ degrees of freedom. We obtain
\[-\log \underset{X\sim P_\phi}{\mathbb{E}}\Big[ e^{-\lambda\sigma^2\norm{\frac{1}{\sigma}X}^2}\Big] 
&= -\log \mathbb{E}_G \bigg[\exp\Big(\frac{-\lambda\norm{G\phi}^2}{1+2\lambda\sigma^2}\Big) 
(1 + 2\lambda \sigma^2)^{-L/2}\bigg] \\[1mm]
&= \frac{\lambda}{1+2\lambda\sigma^2}\norm{\phi}^2
+ \frac{L}{2} \log(1 + 2\lambda \sigma^2). 
\]
We get
\[\infdiv{P_\theta}{P_\phi} \geq \sup_{\lambda\in \R_{\geq 0}}\Bigg\{
 \frac{\lambda}{1+2\lambda\sigma^2}\norm{\phi}^2
+ \frac{L}{2} \log(1 + 2\lambda \sigma^2) - \lambda(1 + \sigma^2L)\Bigg\}.\]
Choose $\lambda = \dfrac{1}{4\sigma^4L}.$ Then $\displaystyle 2\lambda\sigma^2 \leq 1$ so the 
quantity on the right-hand side is at least
\[\frac{\lambda}{2}\norm{\phi}^2 
+ \frac{L}{2}(2\lambda\sigma^2 - 4\lambda^2\sigma^4) - \lambda(1+\sigma^2L)\]
where we have used the fact that $\log(1+x) \geq x - x^2$ for $x \in \R_{\geq 0}$. Since $\norm{\phi} \geq 2$, we have that
\[\infdiv{P_\theta}{P_\phi} &\geq \Big(\frac{\lambda}{4}\norm{\phi}^2 + \lambda\Big) +
(L\lambda \sigma^2
 - 2L\lambda^2\sigma^4) - (\lambda + L\lambda \sigma^2)  \\[1mm]
&= \frac{1}{16 \sigma^4 L}\norm{\phi}^2 - \frac{1}{8\sigma^4L} \geq \frac{1}{32\sigma^4 L}\norm{\phi}^2\]
as desired.
\end{proof}

\subsection{Appendix B}\label{AppendixB}

\begin{proposition}\label{AppendixB1}
Let $\theta,\phi\in \R^L$ be two collision free signals such that $\Delta_2(\theta,\phi) = 0$. 
If $s \geq 7$, then $\phi$ either lies in the orbit of $\theta$ or lies in the orbit of $-\theta$.
\end{proposition}

\begin{proof} We have that for any $1\leq i , j \leq L$ with $i\neq j$,
\[\mathbb{E}[(G\theta)^{\otimes 2}]_{i,j} &= \frac{1}{L}\sum_{g\in \Z_L} \theta(i + g) \theta(j + g)
\\ &= \begin{cases}  0  & \ \text{ if } i-j \not\in \supp(\theta) \\
\dfrac{1}{L}\theta(i')\theta(j') & \ \text{ otherwise, where } i',j'\in \text{supp}(\theta) \text{ with } i - j = i' - j'.
\end{cases}\]
Thus the support of $\theta$ is completely determined by the second moment tensor due to the resolution of the beltway problem. 
As $\Delta_2(\theta,\phi) = 0$, we must have $\supp(\theta) = \supp(R_\ell\phi)$ 
for some $1 \leq \ell \leq L$. Without loss of generality assume $\ell = L$. In other words, we have 
$\text{supp}(\theta) = 
\text{supp}(\phi)$. 
Let $\{k_1,k_2,\cdots,k_s\}$ denote their (common) support. Then the condition 
\[\theta(k_i) \theta(k_{i'}) = \phi(k_i)\phi(k_{i'})  \ \ \ \ \ \ \text{ for all } 1 \leq i,i' \leq s\]
is equivalent to condition that the pairwise product of the ratios
\[\frac{\theta(k_1)}{\phi(k_1)},\ \frac{\theta(k_2)}{\phi(k_2)} , \cdots ,\ \frac{\theta(k_s)}{\phi(k_s)} \]
must all evaluate to $1$. This implies that these ratios must either all be positive or all be negative.  
After reordering if necessary, 
we assume without loss of generality that
\[\bigg|\frac{\theta(k_1)}{\phi(k_1)}\bigg| \geq  \bigg|\frac{\theta(k_2)}{\phi(k_2)}\bigg| 
\geq \cdots \geq \bigg|\frac{\theta(k_s)}{\phi(k_s)}\bigg|. \]
If the ratios are all positive, then since 
\[\frac{\theta(k_1)}{\phi(k_1)}\cdot \frac{\theta(k_2)}{\phi(k_2)} = 1,\] 
we must have $\theta(k_2)/\phi(k_2)\leq 1$. If the above inequality is strict, then 
\[\frac{\theta(k_2)}{\phi(k_2)} \cdot \frac{\theta(k_3)}{\phi(k_3)} < 1,\] a contradiction. Hence 
\[\frac{\theta(k_2)}{\phi(k_2)}\cdot \frac{\theta(k_i)}{\phi(k_i)} = 1 \text{ for all $i \neq 2$}
&\implies 
\frac{\theta(k_i)}{\phi(k_i)} = 1 \ \text{ for all $1\leq i \leq L$} \\[1mm] 
&\implies \theta = \phi.\] 

In the latter case, a similar argument yields that $\theta(k_i)/\phi(k_i) = -1$ for all $1 \leq i \leq L$ and so $\theta = 
-\phi$.
\end{proof}

\subsection{Appendix C}\label{AppendixC}

We give a full proof of the modifield Theorem \ref{PreliminariesKullback2}. 
We first prove a modified version of lemma B.12. Fix a positive number $K_0 \geq 1.$

\begin{lemma}\label{AppendixC1} For any $m\in \Z_{\geq 1}$ 
and $\theta,\phi\in \R^L$ satisfying $\norm{\theta} = 1$ 
and $\rho(\theta,\phi) \leq K_0$, we have 
\[\norm{\Delta_m(\theta,\phi)}^2  
\leq 12 \cdot 18^m K_0^{2m}\rho(\theta,\phi)^2. \]
\end{lemma}

\begin{proof} Assume without loss of generality that $\rho(\theta,\phi) = \norm{\theta - \phi} =: \epsilon.$ 
By Jensen's inequality, 
\[\norm{\mathbb{E}[(G\theta)^{\otimes m} - (G\phi)^{\otimes m}]}^2 
&\leq \mathbb{E}\big[\normnorm{(G\theta)^{\otimes m} - (G\phi)^{\otimes m}}^2\big] 
= \normnorm{\theta^{\otimes m} - \phi^{\otimes m}}^2.\]
Expanding the norm yields
\[\normnorm{\theta^{\otimes m} - \phi^{\otimes m}}^2 &= \norm{\theta}^{2m} 
- 2\langle \theta,\phi \rangle^m + \norm{\phi}^{2m} \\
&= 1 - 2(1+\gamma)^m + (1+2\gamma + \epsilon^2)^m\]
where $\gamma := \langle \theta, \phi - \theta \rangle$ is such that $|\gamma| \leq \epsilon$ 
by Cauchy Schwarz. 

By the binomial theorem, for all $x$ such that $|x| \leq 3K_0^2$, there exists an $r_m$ such that
\[(1+x)^m = \sum_{k=0}^m \begin{pmatrix} m \\ k \end{pmatrix} x^k = 1 + mx + r_m\]
with $|r_m| \leq 18^m K_0^{2m} x^2$. By assumption, we have $|2 \gamma + \epsilon^2| \leq 3K_0^2$ and 
so
\[\normnorm{\theta^{\otimes m} - \phi^{\otimes m}}^2 &\leq 1 - 2 - 2m\gamma + 2\cdot 18^mK_0^{2m}\epsilon^2
+ 1 + 2m\gamma + m\epsilon^2+ 18^m K_0^{2m} \cdot 9\epsilon^2  \\
&= 2\cdot 18^m K_0^{2m} \epsilon^2 + m\epsilon^2 + 18^m K_0^{2m} \cdot 9\epsilon^2 \\
&\leq 12 \cdot 18^m K_0^{2m} \epsilon^2. \]
\end{proof}

We now proof the upper bound of the modified Theorem \ref{PreliminariesKullback2}.

\begin{proof}  If $\theta = 0,$ then since 
$\rho(\theta,\phi) \leq K_0\norm{\theta}$, we must have $\phi = 0$ as well so the statement trivially holds. Otherwise, 
first note that each term in (\ref{PreliminariesKullback2}) remains unchanged if the quantities $\theta$, $\phi$ 
and $\sigma$ are replaced by $\theta/\norm{\theta},\ \phi/\norm{\theta}$ and $\sigma/\norm{\theta}$ 
respectively. The same is true when $\phi$ is replaced by another vector $G_0\phi$ in the same $\mathcal{R}$-orbit. 
As such, we  henceforth assume $\norm{\theta} = 1$, $\sigma \geq 1$ and $\rho(\theta,\phi) = \norm{\theta - \phi}.$

Instead of establishing an upper bound on the KL divergence $\infdiv{P_\theta}{P_\phi}$ directly, we instead work with the $\chi^2$-divergence
\[\chi^2(P_\theta,P_\phi) := 
\int_{\mathbb{R}^d}\frac{(f_\theta(x) - f_\phi(x))^2}{f_\phi(x)}\ dx.\]
and then pass to the KL divergence via the upper bound $\infdiv{P_\theta}{P_\phi} \leq \chi^2(P_\theta,P_\phi)$. Since $\mathbb{E}[G\phi] = 0,$ Jensen's inequality implies that
\begin{align*}f_\phi(x) &\geq
\frac{1}{\sigma^d (2\pi)^{d/2}}e^{-\frac{\norm{x}^2 + \norm{\phi}^2}{2\sigma^2}}e^{\frac{1}{\sigma^2}\mathbb{E}[x^TG\phi]} = 
\frac{1}{\sigma^d (2\pi)^{d/2}}e^{-\frac{\norm{x}^2 + \norm{\phi}^2}{2\sigma^2}}.
\end{align*}
Hence
\[\frac{(f_\theta(x) - f_\phi(x))^2}{f_\phi(x)} \leq 
\frac{1}{\sigma^d(2\pi)^{d/2}}e^{\frac{-\norm{x}^2 + \norm{\phi}^2}{2\sigma^2}}\Big(e^{-\frac{ \norm{\theta}^2}{2\sigma^2}}\mathbb{E}\big[e^{\frac{1}{\sigma^2}x^TG\theta}\big] - e^{-\frac{\norm{\phi}^2}{2\sigma^2}}\mathbb{E}\big[e^{\frac{1}{\sigma^2}x^TG\phi}\big]\Big)^2.\]
To obtain our desired bound on the $\chi^2$-divergence, we will integrate both sides with respect to $x$. Expanding out the square on the right-hand side yield three terms, which we will evaluate separately. The first term is
\begin{align*}
&\ \ \ \ e^{\frac{\norm{\phi}^2 - 2\norm{\theta}^2}{2\sigma^2} }  \int_{\mathbb{R}^d}\frac{1}{\sigma^d (2\pi)^{d/2}}
e^{-\frac{ \norm{x}^2}{2\sigma^2}} 
\mathbb{E}\big[e^{\frac{1}{\sigma^2} x^T(G + G') \theta}\big] \ dx
\end{align*}
where $G'$ denotes an independent and identically distributed copy of $G$. To simplify the expression, we seek to rewrite it as the integral of the density of a Gaussian by completing the square. We obtain
\begin{align*}
&\ \ \ \ e^{\frac{\norm{\phi}^2 - 2\norm{\theta}^2}{2\sigma^2} }  \int_{\mathbb{R}^d}\frac{1}{\sigma^d (2\pi)^{d/2}}
e^{-\frac{ \norm{x}^2}{2\sigma^2}} 
\mathbb{E}\big[e^{\frac{1}{\sigma^2} x^T(G + G') \theta}\big] \ dx \\
&=  e^{\frac{\norm{\phi}^2 - 2\norm{\theta}^2}{2\sigma^2}}\mathbb{E}\Bigg[\int_{\mathbb{R}^d}\frac{1}{\sigma^d (2\pi)^{d/2}}  
e^{-\frac{1}{2\sigma^2} (x-(G + G')\theta)^T (x- (G+G') \theta)}  dx  \cdot e^{\frac{1}{2\sigma^2}((G+ G')\theta)^T((G + G')\theta)} \Bigg] \\ 
&=e^{\frac{\norm{\phi}^2 - 2\norm{\theta}^2}{2\sigma^2}}
\mathbb{E}\Big[e^{\frac{1}{2\sigma^2}((G+ G')\theta)^T((G + G')\theta)} \Big] \\
&=e^{\frac{\norm{\phi}^2}{2\sigma^2}}
\mathbb{E}\big[ e^{\frac{\theta^T G \theta}{\sigma^2}} \big]
\end{align*}
Via similar computations, the second and third terms evaluate to
\[ - 2e^{\frac{\norm{\phi}^2}{2\sigma^2}}
\mathbb{E}\big[ e^{\frac{\theta^T G\phi}{\sigma^2}} \big] \ \ \ \ \ \text{ and } \ \ \ \ \
e^{\frac{\norm{\phi}^2}{2\sigma^2}}
\mathbb{E}\big[ e^{\frac{\phi^TG\phi}{\sigma^2}} \big]\]
respectively. As $\norm{\theta - \phi} \leq K_0$ and $\norm{\theta} = 1$, we have that $\norm{\phi}^2 \leq (K_0 + 1)^2 \leq 4K_0^2$ and so 
\[e^{\frac{\norm{\phi}^2}{2\sigma^2}} \leq  e^{\frac{2K_0^2}{\sigma^2}}.\] A power series expansion yields
\begin{align*}
\chi^2(P_\theta, P_\phi) &\leq e^{\frac{2K_0^2}{\sigma^2}}\mathbb{E}\Big[ e^{\frac{\theta^TG\theta}{\sigma^2}} - 
2e^{\frac{\theta^T G \phi}{\sigma^2}}  + e^{\frac{\phi^T G \phi}{\sigma^2}} \Big] \\
&= e^{\frac{2K_0^2}{\sigma^2}}\sum_{m=1}^\infty \frac{1}{\sigma^{2m}m!} \mathbb{E}\big[ (\theta^T G \theta)^m - 2(\theta^T G \phi)^m + (\phi^T G \phi)^m\big] 
\\
&= 
e^{\frac{2K_0^2}{\sigma^2}}\sum_{m=1}^\infty
\frac{1}{\sigma^{2m}m!}\norm{\Delta_m(\theta,\phi)}^2.
\end{align*}
With the lemma, and using the fact that $\sigma \geq 1$, we conclude the proof with
\[e^{\frac{2K_0^2}{\sigma^2}}\sum_{m=k}^\infty \frac{\norm{\Delta_m(\theta,\phi)}^2}{\sigma^{2m}m!} &\leq 
12e^{\frac{2K_0^2}{\sigma^2}}  \sum_{m=k}^\infty \frac{18^m K_0^{2m} \rho(\theta,\phi)^2}{\sigma^{2m}m!} \\[1mm] &\leq
12e^{\frac{2K_0^2}{\sigma^2}} \cdot \frac{\rho(\theta,\phi)^2}{\sigma^{2k}}\sum_{m=k}^\infty 
\frac{18^m K_0^{2m}}{m!} \\[1mm] &\leq 12e^{\frac{2K_0^2}{\sigma^2} + 18K_0^2} 
\cdot\frac{ \rho(\theta,\phi)^2}{\sigma^{2k}}.\]
\end{proof}

The remainder of the section is dedicated to proving the lower bound of the modified Theorem 
\ref{PreliminariesKullback2}. To that end, we first recall the key ingredients used in the original proof of Theorem 
\ref{PreliminariesKullback2}. At each step, the modifications that are needed to suit our current context are highlighted.

\begin{definition}\label{AppendixC2} The \itbf{(probabilist's) Hermite polynomials} are a family of polynomials 
$\{h_k\}_{k=0}^\infty$ defined by 
\[h_k(x) := (-1)^k e^{\frac{x^2}{2}} \dfrac{\partial^k}{\partial x^k} e^{-\frac{x^2}{2}}, \ \ \ \ \ \ k\in\mathbb{Z}_{\geq 0}.\]
\end{definition}
\begin{fact}\label{AppendixC3} The Hermite polynomials satisfy the following basic properties:
\begin{enumerate}[label=(\roman*)]
\item The polynomial $h_k(x)$ has degree $k$;
\item The family $\{h_k\}_{k=0}^\infty$ is an orthogonal basis for $L^2(\mathbb{R},\gamma)$, where $\gamma$ denotes the standard Gaussian measure on $\mathbb{R};$
\item We have $\norm{h_k}^2_{L^2(\mathbb{R},\gamma)} = k!$;
\item For any $\mu\in\R$, we have $\underset{Y\sim \mathcal{N}(\mu,1)}{\mathbb{E}}\big[h_k(Y)\big]= \mu^k.$
\end{enumerate}
\end{fact}

\begin{definition}\label{AppendixC4}
For each multi-index $\alpha = (\alpha_1,\cdots,\alpha_d)\in \mathbb{N}_0^d$, define the multivariate polynomial $h_\alpha$ by
\[h_\alpha(x_1,\cdots,x_d) := \prod_{i=1}^d h_{\alpha_i} (x_i).\]
The family $\big\{h_\alpha  :  \alpha \in \mathbb{N}_{0}^d \big\}$ is called the \itbf{multivariate Hermite polynomials}.
\end{definition}
The multivariate Hermite polynomials form an orthogonal basis for the product space $L^2(\mathbb{R}^d,\gamma^{\otimes d})$. By properties (ii), (iii) and (iv) of 
Fact \ref{AppendixC3}, the family of rescaled Hermite polynomials
$\{H_\alpha \ : \ \alpha \in \mathbb{N}^d_0\}$ defined by
\begin{equation}\label{AppendixC5}
H_\alpha(x_1,\cdots, x_d) := \sigma^{|\alpha|} h_\alpha(\sigma^{-1}x_1,\cdots,\sigma^{-1}x_d)
\end{equation}
satisfy the following identities
\begin{align}\displaystyle\underset{Z \sim \mathcal{N}(\bs \mu, \sigma^2 \bs{I}_d)}{\mathbb{E}}\big[H_\alpha ( Z)\big] &= \prod_{i=1}^d\mu_i^{\alpha_i} \label{AppendixC6} \\ 
\underset{Z \sim \mathcal{N}(\bs{0},\sigma^2 \bs{I}_d)} {\mathbb{E}}\big[H_\alpha( Z) 
H_\beta( Z) \big] &= \begin{cases} \sigma^{2|\alpha|} \alpha! & \ \text{ if } \alpha = \beta \\
0 & \ \text{ otherwise}.
\end{cases} \label{AppendixC7}
\end{align}
Next, for each positive integer $m$, we define the function $H_m : \R^d \to (\R^d)^{\otimes m}$ in the following way. Given $x\in \mathbb{R}^d$, set $H_m(x)$ to be the order-$m$ symmetric tensor whose $(i_1,\cdots, i_m)$th-entry is given by
$H_{\alpha^{(i_1\cdots i_m)}}(x)$, where $\alpha^{(i_1\cdots i_m)} \in \{0,\cdots,m\}^d$ is the multi-index associated to $(i_1,\cdots,i_m)$:
\[\alpha^{(i_1 \cdots i_m)}_\ell := \big|\big\{{j \in [m] \ : \ i_j = \ell}\big\}\big|, \ \ \ \ \ \ \ 1\leq \ell \leq d.\]
Note that  $|\alpha^{(i_1\cdots i_m)}| = m$ for each $m$-tuple $(i_1,\cdots,i_m)\in [d]^m$.

The motivation behind the above definition will gradually become apparent once we write the quantities  $\norm{\Delta_m(\theta,\phi)}^2$ in terms of the family $(H_m)_{m=1}^\infty$, where $\theta$ and $\phi$ are arbitrary vectors in $\R^d$. For a positive integer $k$, consider the degree $k$ polynomial
\[T_k(x) := \sum_{m=1}^k \frac{\langle \Delta_m(\theta,\phi) , H_m(x) \rangle}{(\sqrt{3} \sigma)^{2m} m!}.\]
If $X\sim P_\theta$, then (\ref{AppendixC5}) implies that
\[\mathbb{E}[T_k(X)] = 
\mathbb{E}_G\Bigg[\sum_{m=1}^k \frac{\big\langle \Delta_m(\theta,\phi) , \mathbb{E}_\xi[H_m(X)] \big\rangle}{(\sqrt{3} \sigma)^{2m} m!} \Bigg]
= \sum_{m=1}^k \frac{\big\langle 
\Delta_m(\theta,\phi), \mathbb{E}[(G\theta)^{\otimes m} ]\big\rangle}{(\sqrt{3}\sigma)^{2m} m!}.\]
Hence if $Y \sim P_\phi$, we get
\[\mathbb{E}[T_k(X)] - \mathbb{E}[T_k(Y)] 
= \sum_{m=1}^k \frac{\norm{\Delta_m(\theta,\phi)}^2}{(\sqrt{3}\sigma)^{2m} m!} =: \delta.\]
To proceed, we will use the following lemma to relate the KL divergence between $P_\theta$ and $P_\phi$ to the quantity $\delta$. 
(Remark: This lemma remains unchanged from lemma B.11 of the BRW paper).

\begin{lemma}\label{AppendixC8}
Let $P_1$ and $P_2$ be any two probability distributions on a measure space $(\Omega,\mathcal{F})$. Suppose that there exists a measurable function $F : \Omega \to \R$ such 
that $\big(\mathbb{E}_{P_1}[F(X)] - \mathbb{E}_{P_2}[F(X)]\big)^2 = \mu^2$ and 
$\max\big\{\text{var}_{P_1}(F(X)),\text{var}_{P_2}(F(X))\big\} \leq \sigma^2$. Then 
\begin{equation}\label{AppendixC9}\infdiv{P_1}{P_2} \geq \frac{\mu^2}{4\sigma^2 + \mu^2}.\end{equation}
\end{lemma}

\begin{proof} By replacing $F$ by $F + \lambda$ for a suitably chosen constant $\lambda$, we may assume that $\mathbb{E}_{P_1}[F(X)] = \mu/2$ and $\mathbb{E}_{P_2}[F(X)] = -\mu/2$. Let $Q_1$ and $Q_2$ denote the corresponding probability distributions of $F(X)$ when $X$ is distributed according to $P_1$ and $P_2$ respectively. By the data processing inequality, it suffices to prove the claimed bound for $\infdiv{Q_1}{Q_2}.$ We further assume that $Q_1$ is absolutely continuous with respect to $Q_2$ (otherwise the bound is trivial).

As the quantities involved in
(\ref{AppendixC9}) arise from taking the expectation of random variables, our approach to establish the inequality will be to pass to the convex function $f : [0,+\infty) \to \R$ defined by
\[f(x) := x\log x - \frac{(x-1)^2}{2(x+1)}\]
and then apply Jensen's inequality. This yields
\[\mathbb{E}_{Q_2}\Bigg[f\left(\frac{d Q_1}{dQ_2}\right)\Bigg] \geq f\Bigg(\mathbb{E}_{Q_2}\bigg[\frac{d Q_1}{d Q_2}\bigg]\Bigg) = f(1) = 0.\]
Let $\mu$ be a dominating measure (i.e. $Q_1 \ll \mu$ and $Q_2 \ll \mu$) and let $q_1$ and $q_2$ denote the densities of $Q_1$ and $Q_2$ with respect to $\mu$. The previous calculation implies that
\[\infdiv{Q_1}{Q_2} = \mathbb{E}_{Q_2} 
\bigg[\frac{dQ_1}{dQ_2} \log \frac{dQ_1}{dQ_2}\bigg]\geq \frac{1}{2} \int_\R 
\frac{(q_1(x) - q_2(x))^2}{(q_1(x) + q_2(x))}\ d\mu(x).\]
By the Cauchy-Schwarz inequality,
\begin{align*}
\mu^2 &= \Bigg(\int_\R x \big(q_1(x) - q_2(x)\big)\ d\mu(x)\Bigg)^2 \\
&\leq \int_\R x^2\big(q_1(x) + q_2(x)\big)
\ d\mu(x) \int_\R \frac{(q_1(x) - q_2(x))^2}{q_1(x) + q_2(x)}\ d\mu(x) \\
&= (2\sigma^2 + \mu^2/2) \int_\R \frac{(q_1(x) - q_2(x))^2}{q_1(x) + q_2(x)}\ d\mu(x).
\end{align*}
Hence 
\[\infdiv{Q_1}{Q_2} \geq \frac{\mu^2}{4\sigma^4 + \mu^2}\]
as desired.
\end{proof}

With the above lemma, our strategy for establishing lower bounds for the KL divergence will be to lower bound the quantity $\delta$ and upper bound the  variances of both $T_k(X)$ and $T_k(Y)$. To control $\delta$, we will apply Lemma \ref{AppendixC1}. On the other hand, to control the variances, we will use its Hermite decomposition as a gateway to bring in heavy machinery from harmonic analysis. The following lemma remains unchanged from Lemma B.13 of the BRW paper.
\begin{lemma}\label{AppendixC10}
Fix a positive integer $k$. Let $\zeta\in\R^d$ and suppose that $Y\sim P_\zeta$. Then for any symmetric tensors $S_1, \cdots, S_k$, where $S_m \in (\R^d)^m$, we have
\[\text{var}\left(\sum_{m=1}^k \frac{\langle S_m, H_m(Y) \rangle}{(\sqrt{3}\sigma)^{2m} m!}\right) \leq e^{\frac{\norm{\zeta}^2}{2\sigma^2}} \sum_{k=1}^k \frac{\norm{S_m}^2}{(\sqrt{3}\sigma)^{2m} m!}.\]
\end{lemma}

\begin{proof} Let $F(x) = \sum_{m=1}^k \frac{\langle S_m, H_m(x) \rangle}{(\sqrt{3}\sigma)^{2m}m!}$. To upper bound the variance, it suffices to upper bound the second moment $\mathbb{E}[F(Y)^2]$. Before we are able to bring in results from the theory of Gaussian spaces, we first need to replace $P_\zeta$ with the centered multivariate normal distribution $Z\sim P_0 = \mathcal{N}(\bs 0, \sigma^2\bs I_d).$ To that end, we apply the Cauchy-Schwartz inequality to obtain
\begin{align}
\mathbb{E}[F(Y)^2] &= \int_{\R^d} f_\zeta(x) F(x)^2\ dx \nonumber\\[1mm]
&\leq  \left(\int_{\R^d} f_0(x) F(x)^4 \ dx\right)^{1/2}  \left(\int_{\R^d} \frac{f_\zeta(x)^2}{f_0(x)}\ dx\right)^{1/2} \nonumber \\[1mm]
&= \mathbb{E}\big[F(Z)^4\big]^{1/2} \left(\int_{\R^d} \frac{f_\zeta(x)^2}{f_0(x)}\ dx\right)^{1/2}. \label{AppendixC11}
\end{align}
We first address the second term. This is done by proceeding in a similar fashion as in the proof of the upper bound of the modified Theorem 9. Observe that
\begin{align*}\int_{\R^d} \frac{f_\zeta(x)^2}{f_0(x)}\ dx  
&= 
\frac{1}{\sigma^d (2\pi)^{d/2}}\int_{\R^d} 
\frac{\mathbb{E}\big[\exp(-\frac{1}{2\sigma^2}(\norm{x}^2 - 2x^T G\zeta + \norm{\zeta}^2))\big]^2}{\exp(-\frac{1}{2\sigma^2}\norm{x}^2)}\ 
dx.
\end{align*}
By applying Jensen's inequality and then completing the square afterwards, we obtain
\begin{align}
&\ \ \ \  \int_{\R^d} \frac{f_\zeta(x)^2}{f_0(x)}\ dx \nonumber\\[1mm]  &\leq 
\mathbb{E}\Bigg[\frac{1}{\sigma^d (2\pi)^{d/2}}\int_{\R^d} 
\frac{\exp\big(-\frac{1}{\sigma^2}(\norm{x}^2 - 2x^T G\zeta + \norm{\zeta}^2)\big)}{\exp(-\frac{1}{2\sigma^2}\norm{x}^2)}\ 
dy.\Bigg] \nonumber\\[1mm]
 &=
\mathbb{E}\Bigg[\frac{1}{\sigma^d (2\pi)^{d/2}}\int_{\R^d} 
\exp\left(-\frac{1}{2\sigma^2} \norm{x}^2 + \frac{2}{\sigma^2} x^T G - \frac{2}{\sigma^2} \norm{\zeta}^2\right)\ dy \cdot \exp\bigg(\frac{1}{\sigma^2}\norm{\zeta}^2\bigg)\Bigg] \nonumber\\[1mm]
&= \exp\bigg(\frac{1}{\sigma^2}\norm{\zeta}^2\bigg) \label{AppendixC12}.
\end{align}
We now come to the crux of the matter, which is to establish an upper bound on the first term $(\mathbb{E}[F(Z)^4])^{1/2}$. To accomplish that goal, we bring in some standard results about the \itbf{Ornstein-Uhlenbeck semigroup}, which is a family of operators $U_\rho : L^2(\R^d, \gamma^{\otimes d}) \to L^2(\R^d, \gamma^{\otimes d})$ defined by
\[U_\rho(f)(z) := \underset{Z' \sim \mathcal{N}(\bs 0,\sigma^2 \bs{I}_d)}{\mathbb{E}} \Big[ f\big(\rho z + \sqrt{1-\rho^2}\cdot Z'\big)\Big],\ \ \ \ \ \ \rho \in [-1,1].\]
Here, we highlight that our definition of $U_\rho$ differs from the standard definition in the literature in the sense that expectation is taken with respect to a multivariate normal distribution with covariance matrix $\sigma^2 \bs{I}_d$ as opposed to $\bs{I}_d$ to compensate for the scaling in $(\ref{AppendixC5}).$

The set $\{H_\alpha\ : \ \alpha\in \mathbb{N}^d_0\}$ is an eigenbasis for the family $(U_\rho)$, with
$U_\rho(H_\alpha) = \rho^{|\alpha|} H_\alpha$ 
\cite[Proposition 11.37]{AnalysisOfBooleanFunctions}. By viewing $\langle S_m, H_m(x) \rangle$ as a polynomial in $x$, we get
\begin{align*}
U_\rho\big(\langle S_m, H_m(x) \rangle\big) &= U_\rho \left(\sum_{1\leq i_1 ,\cdots ,i_m \leq d} (S_m)_{i_1\cdots i_m} (H_m)_{i_1\cdots i_m}\right) 
\\[1mm] &=  \sum_{1\leq i_1 ,\cdots, i_m \leq d} (S_m)_{i_1\cdots i_m} U_\rho\big(H_{\alpha^{(i_1\cdots i_m)}}\big) \\[1mm]
&= \rho^m\sum_{1\leq i_1 ,\cdots ,i_m \leq d} (S_m)_{i_1\cdots i_m} 
H_{\alpha^{(i_1\cdots i_m)}} = \rho^m \langle S_m, H_m(x) \rangle,
\end{align*}
where we have used the fact that $|\alpha^{(i_1\cdots i_m)}| = m$. Thus if we define the degree $k$ polynomial
\[\widetilde{F}(x) := \sum_{m=1}^k \frac{\langle S_m , H_m(x) \rangle}{(\sqrt{3})^m \sigma^{2m} m!},\] then $U_{1/\sqrt{3}}(\widetilde{F}) = F.$ 
Next, we will invoke the Gaussian hypercontractivity theorem:
\begin{theorem}\cite[Theorem 11.23]{AnalysisOfBooleanFunctions}\label{AppendixC13} Let $1\leq p \leq q \leq \infty$ 
and let $f\in L^p(\R^n,\gamma)$, where $\gamma$ is the standard Gaussian measure on $\R^n$. Then $\norm{U_{\rho} f}_q \leq \norm{f}_p$ for $0 \leq \rho \leq \sqrt{\frac{p-1}{q-1}}$.
\end{theorem}
In our current context, the Gaussian hypercontractivity theorem implies that
\[\mathbb{E}\big[F(Z)^4\big]^{1/2} \leq \mathbb{E}\big[\widetilde{F}(Z)^2\big].\]
It remains to compute $\mathbb{E}[\widetilde{F}(Z)^2]$. Due to the orthogonality relations in 
(\ref{AppendixC7}), when we expand the square, most of the terms will vanish. Since $S_m$ and $H_m(x)$ are both symmetric tensors, for any tuple $(i_1,\cdots,i_m) \in [d]^m$, the quantities $(S_m)_{i_1\cdots i_m}$ and 
$(H_m(x))_{i_1\cdots i_m}$ depend only on the multi-set $\{i_1,\cdots,i_m\}.$ Thus for each $\alpha \in \mathbb{N}^d_0$ such that $|\alpha | = m$, if we define \[S_\alpha := (S_m)_{i_1\cdots i_m}\]
where $(i_1,\cdots,i_m)$ is any $m$-tuple satisfying $\alpha^{(i_1\cdots i_m)} = \alpha$, we have that
\[\langle S_m , H_m(x) \rangle = \sum_{\substack{|\alpha| = m}} \frac{m!}{\alpha!} S_\alpha H_\alpha(x). \]
Applying the orthogonality relations in (\ref{AppendixC7}), we obtain
\begin{align}\mathbb{E}\big[F(Z)^4\big]^{1/2} \leq \mathbb{E}\big[\widetilde{F}(Z)^2\big] &= \sum_{m=1}^k 
\mathbb{E}\Bigg[ \bigg(\frac{1}{(\sqrt{3})^m \sigma^{2m} m!} \sum_{\substack{ |\alpha| = m}} \frac{m!}{\alpha!} S_\alpha H_\alpha(Z) \bigg)^2 \Bigg] \nonumber\\ &= \sum_{m=1}^k \frac{1}{3^m \sigma^{4m}}
\sum_{\substack{ |\alpha| = m}} \frac{S_\alpha^2}{\alpha!^2}\cdot 
\mathbb{E}\big[H_\alpha(Z)^2\big] \nonumber\\
&= \sum_{m=1}^k \frac{1}{(\sqrt{3} \sigma)^{2m} m!}
\sum_{\substack{ |\alpha| = m}} \frac{m!}{\alpha!} S_\alpha^2 \nonumber\\
&= \sum_{m=1}^k \frac{\norm{S_m}^2}{(\sqrt{3} \sigma)^{2m} m!}. \label{AppendixC14}
\end{align}
Plugging in (\ref{AppendixC12}) and (\ref{AppendixC14}) back into (\ref{AppendixC11}) gives the desired conclusion. 
\end{proof}
We now conclude our discussion in this section by putting together everything that was introduced. 
Note that while the constants in our proof differ from the original proof given in \cite[Theorem 9]{BRW}, 
the spirit of the proof remains unchanged.

\begin{proof}   As in the proof of the upper bound, we assume $\rho(\theta,\phi) = \norm{\theta - \phi},\ \norm{\theta} = 1$ and $\sigma \geq 1$. By 
Lemma \ref{AppendixC1}, we have
\[\delta =  \sum_{m=1}^k \frac{\norm{\Delta_m(\theta,\phi)}^2}{(\sqrt{3}\sigma)^{2m} m!} 
\leq 12 \rho(\theta,\phi)^2\sum_{m=0}^\infty \frac{18^m \cdot K_0^{2m}}{3^mm!} \leq 12 K_0^2 e^{6K_0^2}.\]
By Lemma \ref{AppendixC10}, since $\norm{\theta}^2 = 1$ and $\norm{\phi}^2 \leq (K_0 + 1)^2 \leq  4K_0^2$, the variances of 
$T_k(X)$ and $T_k(Y)$ are bounded above by $e^{\frac{2K_0^2}{\sigma^2}}\delta$. 
Now since $4e^{\frac{2K_0^2}{\sigma^2}} \leq 12 K_0^2 e^{6K_0^2}$
applying Lemma \ref{AppendixC8} then gives
\[\infdiv{P_\theta}{P_\phi} \geq \frac{\delta^2}{4e^{\frac{2K_0^2}{\sigma^2}}\delta + \delta^2} \geq \frac{\delta^2}{24K_0^2 e^{6K_0^2} \delta}= 
\frac{1}{24}K_0^{-2} e^{-6K_0^2} \sum_{m=1}^k \frac{\norm{\Delta_m(\theta,\phi)}^2}{(\sqrt{3}\sigma)^{2m} m!}.\]
Finally, as the summands are nonnegative, letting $k\to\infty$ gives the desired result.
\end{proof}

\end{document}